\documentclass[12pt]{amsart}

\usepackage[english]{babel}
\usepackage[T1]{fontenc}
\usepackage{amsmath,amssymb}

\usepackage[a4paper,top=3cm,bottom=2cm,left=3cm,right=3cm,marginparwidth=1.75cm]{geometry}

\usepackage{amsthm}
\usepackage{amsfonts}
\usepackage{graphicx}
\usepackage[colorinlistoftodos]{todonotes}
\usepackage[colorlinks=true, allcolors=blue]{hyperref}
\usepackage{url}
\newtheorem{theorem}{Theorem}[section]
\newtheorem{proposition}[theorem]{Proposition}
\newtheorem{definition}[theorem]{Definition}
\newtheorem{lemma}[theorem]{Lemma}
\newtheorem{remark}[theorem]{Remark}

\newtheorem{conjecture}[theorem]{Conjecture}
\newtheorem{corollary}[theorem]{Corollary}
{
      \theoremstyle{plain}
      
  }
  {
      \theoremstyle{plain}
      
  }

\def\FF{\mathbb{F}}
\def\KK{\mathbb{K}}
\def\NN{\mathbb{N}}
\def\QQ{\mathbb{Q}}

\def\ZZ{\mathbb{Z}}

\def\calo{\mathcal{O}}

\def\Frob{\mathrm{Frob}}
\def\Gal{\mathrm{Gal}}
\def\GL{\mathrm{GL}}
\def\PGL{\mathrm{PGL}}
\def\M{\mathrm{M}}

\def\SL{\mathrm{SL}}

\def\Qbar{\overline{\QQ}}

\usepackage{tikz-cd,enumerate,ulem}
\usepackage{comment}
\usetikzlibrary{arrows}
\usepackage[utf8]{inputenc}

\newtheorem{assumptions}[theorem]{Assumption}
\newcommand{\Z}{{\mathbb{Z}}}
\newcommand{\Q}{{\mathbb{Q}}}
\newcommand{\Qb}{\overline{\mathbb{Q}}}
\newcommand{\im}{{\mathrm{im}}}
\newcommand{\Gm}{{\mathbb G}_{\rm m}}
\newcommand\B{{\rm (B)}}
\newcommand\I{{\rm I}}
\newcommand\No{N}%{\mathcal{N}} %livello
\renewcommand\O{{\mathcal O}} %Anello interi
\newcommand{\mug}{{\boldsymbol\mu}}

\newcommand\e{{\varepsilon}}
\newcommand\mA{{\mathcal A}}
\newcommand\mB{{\mathcal B}}
\newtheorem*{Remark 2.2}{Remark 2.2}
\newtheorem*{theorem*}{Theorem}
\newtheorem*{remark*}{Remark}
\newtheorem*{fact*}{Fact} 
\newcommand{\stkout}[1]{\ifmmode\text{\sout{\ensuremath{#1}}}\else\sout{#1}\fi}
\newcommand\K{{\mathbb{K}}}
\newcommand\F{{\mathbb{F}}}
\renewcommand\L{{\mathbb{L}}}
\newcommand\E{{\mathbb{E}}}
\newcommand\El{{\mathcal{E}}}
\numberwithin{equation}{section}
\title{Bogomolov property and Galois representations}%Bogomolov property for some modular Galois representations}

\author{Francesco Amoroso$^{(1)}$ and Lea Terracini$^{(2)}$}
\date{}
\begin{document}

\maketitle
\vskip -0.5cm
\begin{center}
{\small\it
$^{(1)}$Dipartimento di Matematica\\
Universit\`a di Torino\\
Via Carlo Alberto, 10 - 10123 Torino, Italy\\[0.3cm]
$^{(2)}$Dipartimento di Informatica\\
Universit\`a di Torino\\
Corso Svizzera, 185 - 10149 Torino, Italy}
\end{center}

\begin{abstract}
In 2013 P. Habegger proved the Bogomolov property for the field generated over $\Q$ by the torsion points of a rational elliptic curve. We explore the possibility of applying the same strategy of proof to the case of field extensions cut out  by Galois representations arising from more general modular forms. 
\end{abstract}



\section{Introduction}
\label{intro}
Let $\alpha$ be an algebraic number of degree $d$. We denote by $h(\alpha)$ the absolute and logarithmic Weil height. Assuming that $\alpha$ is not a root of unity, D. H. Lehmer suggested that $dh(\alpha)$ can be bounded below uniformly in $\alpha$. Known as \lq\lq Lehmer's problem\rq\rq, this question is still open. Nevertheless, in some special cases, not only Lehmer's conjecture is true, but it can also be sharpened. This different point of view was considered for the first time in \cite{BombieriZannier2001}, where the authors introduced the notion of fields with Bogomolov property, (property \B{} for short), where the height is {\sl uniformly  bounded} below outside torsion points. By Northcott's Theorem, number fields are trivial examples of fields with \B{}. The compositum $\Q^{\rm tr}$ of all totally real fields is a less trivial example (apply~\cite[Corollary $1^\prime$]{Schinzel1974} to the linear polynomial $P(z)=z-\alpha$ to get the optimal lower bound). In~\cite{BombieriZannier2001} the authors also show that fields with bounded local degrees at some finite place has Property \B{}. Note that by a result of~\cite{Checcoli2013} a Galois extension has uniformly bounded local degrees at {\sl every} finite place if and only if its Galois group has finite exponent. Thus (infinite) Galois extensions with Galois group of bounded exponent have Property \B{}.

There are still other several interesting examples of subfields of the algebraic numbers with Property \B{}. Zannier conjectured that the maximal abelian extension $\Q^{\rm ab}$ of the rational field also satisfies property \B{}, and this was proved in \cite{AmorosoDvornicich2000}.  By a special case of \cite[Theorem 1.1]{AmorosoZannier2000}, the same is true for $\K^{\rm ab}$, for any number field $\K$. Moreover we can combine this result with the quoted result on Galois extensions of bounded exponent. Let $\K$ be a number field and let $\L/\K$ be an infinite Galois extension with Galois group $G$ with $G/Z(G)$ of finite exponent. Then $\L$ has Property \B{} (\cite[Corollary 1.7]{AmorosoDavidZannier2014}).

Another class of fields with Property \B{} has been exhibited in \cite{Habegger2013}. Let $\El$ be an elliptic curve defined over \(\Q\). Then the field $\Q(\El_{tors})$ obtained by adjoining all torsion points of $\El$ has the Bogomolov property. If $\El$ has complex multiplication then $\Q(\El_{tors})$ is an abelian extension of a quadratic number field and thus it has the property \B{} by the quoted result of \cite{AmorosoZannier2000}. However for a non-CM elliptic curve there are no number fields $\K$ such that $\Q(\El_{tors})$ is contained in the maximal abelian extension of $\K$. \par
Note that by Kronecker-Weber theorem $\Q^{\rm ab}$ is the field obtained by adjoining all torsion points of the multiplicative group. A conjecture due to S. David (see \cite[Conjecture, p.114]{Frey2021}) suggests that $\Q(A_{tors})$ has Property \B{}, when $A$ is an abelian variety defined on a number field, according as it is already known for the algebraic groups $\Gm$ and $\El=$ elliptic curve. Note that for a CM abelian variety $A$ the field $\Q(A_{tors})$ is still an abelian extension of a (CM) number field, and thus property \B{} holds, again by the quoted result of \cite{AmorosoZannier2000}.  

Habegger’s method provides a lower bound $h(\alpha)\geq c(p)$ with $c(p)>0$ whenever we have a prime $p\geq 5$ such that $\El$ has supersingular reduction at $p$ (assumption $(P1)$) and the Galois representation $\Gal(\Qb/\Q)\to {\rm Aut}(\El[p])$ is surjective (assumption $(P2)$). See \cite[Theorem 1.3, p.114]{Frey2021} for an explicit result. By a result of Elkies there exist infinitely many primes satisfying $(P1)$; and by a result of Serre $(P2)$ holds outside a finite set of primes. Thus we can find a prime $p$ satisfying both the above conditions. 
Let now A be an abelian variety. Even under similar assumptions for a prime $p$, an analogue of the result of Habegger is not known.\\

One of the main purposes of this article is to re-interpret some known results on property \B{} in terms of Galois  representations, with the aim of providing a unifying framework for a part of the theory.

\begin{definition}
Let $\Omega$ be a separated topological group, $\KK$ be a subfield of $\Qbar$ and $G_\KK=\Gal(\Qbar/\KK)$. We say that a continuous homomorphism 
$$
\rho: G_\KK \longrightarrow \Omega,
$$
has property \B{} if $\KK(\rho):=\Qbar^{\ker(\rho)}$ does. 
\end{definition}
 
From this point of view,  the main result  of \cite{AmorosoDvornicich2000} implies (by Galois theory) that  every homomorphism of $G_\Q$ taking values in an abelian group has Property \B{}. 

In this vein,~\cite[Corollary 1.7]{AmorosoDavidZannier2014} has the following consequence. Let $\K$ be a subfield of $\Qbar$, $\kappa$ be any field, $\rho:G_\K\to \GL_n(\kappa)$ be an absolutely irreducible representation, and $\L=\KK(\rho)$. Schur's Lemma asserts that every intertwining automorphism of  $\rho$ is a scalar operator. In particular the image of the center of $G=\Gal(\L/\K)$ is contained in the subgroup of scalar matrices of $\GL_n(\kappa)$, so that  the reduction map  $\tilde\rho: G\to \PGL_n(\kappa)$ factors through $G/Z(G)$. Thus, if the image of $\tilde\rho$ has finite exponent, the same is true 
for $G/Z(G)$ and~\cite[Corollary 1.7]{AmorosoDavidZannier2014} implies that $\rho$ has property \B{}.

Habegger theorem \cite[Theorem 1]{Habegger2013} establishes Property \B{} for the map
$$
\rho_\El: G_\Q \longrightarrow \GL_2(\widehat\ZZ)
$$ 
which is the product of all the $\ell$-adic representations
$$
\rho_{\El,\ell}: \Gal(\Qbar/\Q) \longrightarrow Aut(T_\ell(\El))=\GL_2(\ZZ_\ell)
$$
where $\El$ is an elliptic curve defined over $\Q$ and $T_\ell(\El)$ is the corresponding Tate module. Indeed, by Galois correspondence $\Q(\El[\ell^n])$ is the subfield of $\Qbar$ fixed by the kernel of the action of $G_\Q$ on the $\ell^n$-torsion points of $\El$; by passing to the projective limit we see that $\Q(\El[\ell^\infty])=\QQ(\rho_{\El,\ell})$. Since $\Q(\El_{tors})$ is the compositum of $\Q(\El[\ell^\infty])$,   for every prime $\ell$, then it coincides with the  field fixed  by $\bigcap_\ell\ker(\rho_{\El,\ell})=\ker(\rho_{\El})$.\par

By the Shimura-Taniyama-Weil conjecture, proved by Wiles, all elliptic curves defined over $\Q$ are associated to a cuspidal eigenform of weight $2$ on $\Gamma_0(\No)$, whose eigenvalues for the Hecke operators are rational numbers. A question that naturally arises starting from Habegger's theorem is whether and to what extent it can be generalized to other Galois representations associated with modular forms. Let $k\geq 2$ be a positive integer and $f\in S_k(\Gamma_0(\No))$  be a normalized eigenform for the Hecke algebra. Let $\K_f$ be the number field generated by the Hecke eigenvalues of $f$, and  let $\calo_f$ be its ring of integers. A construction due to Shimura and Deligne attaches to $f$ a family of compatible Galois representations 
$$
\rho_{f,v}:G_\Q \longrightarrow \GL_2(\calo_{f,v})
$$
indexed by the finite places $v$ of $\calo_f$.\par
For a rational prime $p$, we let 
$$
\rho_{f,p}:G_\Q \longrightarrow \prod_{v|p} \GL_2(\calo_{f,v} ).
$$
Moreover, let 
$$
\rho_f=\prod_{v}\rho_{f,v}:G_\Q\longrightarrow \GL_2(\widehat\calo_f),
$$
where $\widehat\calo_f$ is the profinite completion of $\calo_f$.\par
Inspired by Habegger's theorem, we venture the following:
\begin{conjecture}
\label{conj:main} 
Let $f\in S_k(\Gamma_0(\No))$ be a normalized eigenform. Then,  
\begin{itemize}
    \item[a)]  for every $v$, $\rho_{f,v}$ has property \B{};
    \item[b)] $\rho_{f}$ has property \B{}.
    \end{itemize}
\end{conjecture}
Notice that Conjecture \ref{conj:main}a) does not imply Conjecture \ref{conj:main}b), as property \B\  is not preserved under field composition (for instance, $\Q^{\rm tr}$ and obviously also $\Q(i)$ have property \B{} but not $\Q^{\rm tr}(i)$, see~\cite[Theorem 5.3, p. 1902]{AmorosoDavidZannier2014}). 

If $f$ is a CM modular form, then  $\QQ(\rho_f)$ is an abelian extension of a quadratic field, (see for instance \cite[Proposition 4.4]{Ribet1977}). Thus $\rho_{f}$ has property \B{} by the special case $D=1$ of~\cite[Theorem 1.1]{AmorosoZannier2000}.

When $k=2$, Shimura \cite{Shimura1973} associated to $f$ an abelian variety $A$ defined over $\Q$, of dimension $[\K_f:\Q]$ equipped with an action of $\calo_f$; such abelian varieties have been called {\sl of type $\GL_2$} in \cite{Ribet1997} and a generalization of Wiles' theorem~\cite[Corollary 10.2]{KhareWintenberger2009} establishes that all of them arise from modular forms. 

Let $A$ be the abelian variety associated to a modular form $f$ of weight $k=2$. Then $\Q(\rho)=\Q(A_{tors})$. Indeed by costruction, for every prime $p$ the representation $\rho_{f,p}=\prod_{v|p} \rho_{f,v}$ is given by the action of $G_\QQ$ on the $p$-adic Tate module $T_p(A)=\varprojlim_{n}A[p^n]$. Then it is immediate to see that $\QQ(A[p^\infty])$ is the fixed field by $\ker\rho_{f,p}$, so that $\QQ(\rho_{f,p})=\QQ(A[p^\infty])$. Passing to $\rho_f$, we have 
$$
\Q\left (\rho_f\right )=\Q\left (\prod_p\rho_{f,p}\right )=\coprod_p\Q\left (A[p^\infty]\right )=\Q \left (\bigcup_{p}A[p^\infty]\right)=\Q \left (\sum_{p}A[p^\infty]\right ),
$$
where the last equality holds because the sum in $A$ is given by rational functions over $\QQ$.

Thus Conjecture~\ref{conj:main} intersects the already mentioned conjecture of David, which predicts that $\Q(A_{tors})$ has Property \B{}, when $A$ is an abelian variety defined on a number field.\\

Assume that $\sum_{n\geq 1}a_nq^n$ is the $q$-expansion of $f$. Let $p$ be a rational prime. Having fixed $f$, we consider the following assumption on $f$ and $p$:
\begin{assumptions}
\label{ass:main}~
\begin{itemize}
\item[$(P0)$] $p\nmid \No$;
\item[$(P1)$] $a_p=0$;
\item[$(P2)$]  the reduction $G_\Q\to \GL_2(\calo_f/p\calo_f)$ of $\rho_{f,p}$  has image\footnote{The image is always contained in 
$\widehat{G}(\calo_f/p\calo_f)$, see the remark above Assumption~\ref{ass:main}.}
$$
\widehat{G}(\calo_f/p\calo_f)=\{\gamma\in \GL_2(\calo_f/p\calo_f)\ |\ \det\gamma\in(\FF_p^\times)^{k-1}\}.
$$
\item[$(P3)$]\!\!\footnote{Note that all these conditions are satisfied if $p\geq 2k-1$.}  $p\geq 5$, $p\nmid k-1$ and $\frac{p+1}2\nmid k-1$.\end{itemize} 
\end{assumptions}
Our main result is the following:
\begin{theorem}
\label{teo:main1}
Let $f=\sum_{n\geq 1} a_nq^n \in S_k(\Gamma_0(\No))$ be a normalized eigenform. Assume that there exists  a rational prime $p$ which satisfies Assumption~\ref{ass:main}. 
Let $\F/\QQ$ be a Galois extension unramified at $p$ with property \B. Then $\rho_{f,p}|_{G_\F}$ has property \B.
\end{theorem}

As a consequence (see section~\ref{sect:proofs} for the proof), we obtain the following direct generalization of Habegger's result:

\begin{theorem}
\label{teo:main}
Let $f=\sum_{n\geq 1} a_nq^n \in S_k(\Gamma_0(\No))$ be a normalized eigenform. Assume that there exists a rational prime $p$ satisfying Assumption~\ref{ass:main}. Then $\rho_f$  has property \B.
\end{theorem}

 In Section~\ref{sect:proofs} we also prove a more general statement (Corollary \ref{cor:genplessis}),  claiming that property \B\  holds for $\FF(\rho_f)$, where $\FF$ is any Galois algebraic field  unramified at $p$ and such that a power of Frobenius is central in $\Gal(\FF/\QQ)$.

Recall that a modular form $f$ is  {\sl supersingular} at a prime $\mathfrak{p}$ of $\calo_f$ above $p$ if $a_p\in\mathfrak{p}$. When $a_p\in\ZZ$ (as in the case $\K_f=\Q$) and $p\geq 5$, this implies $a_p=0$, by the Deligne bound: $|a_p|\leq 2\sqrt{p}$. Therefore in the case of elliptic curves, Habegger's assumption $p\geq 5$ supersingular is equivalent to $a_p=0$. Note that when $\K_f\not=\Q$ the condition $a_p=0$, required in our proof, is in general stronger than supersingularity.\par
Our condition $(P2)$ generalizes Habegger's assumption on the surjectivity of the Galois representation $G_\Q\to {\rm Aut}(\El[p])$.  It has been extensively investigated in the literature (see \cite{Ribet1985, Ribet1997, Arai2012, GhateParent2012,Lombardo2016}); we summarize here the most relevant results for the purposes of this work. 

In the case of modular forms associated to a non CM elliptic curve, a theorem of Serre \cite{Serre1972} establishes property $(P2)$ for every $p$ except finitely many. This  result, together with the aforementioned Elkies'theorem, ensures the non conditional character of Habegger's theorem.

For a modular form such that $\KK_f\not=\QQ$  the situation concerning the image of $\rho_{f,p}$ is complicated  by the possible of presence of {\sl inner twists}; i.e. the elements of  the subgroup $\Gamma \subseteq \mathrm{Aut}(\KK_f)$
 consisting of those automorphisms $\sigma$
 for which $f^\sigma$
 is equal to the twist of $f$
 by some Dirichlet character. Let $\KK_1$ be the subfield of $\KK$ fixed by $\Gamma$.
It is proved in \cite[Theorem 3.1]{Ribet1985} that there is an abelian number field $E$ such that for $p$ large enough relatively to $f$,  the restriction $\rho_{f,p}|_{G_E}$ has  image $\widehat G(\calo_{\KK_1}\otimes \ZZ_p)$.

Therefore, all we can say is that condition $(P2)$ holds for $p$ large enough when $f$ has no inner twists. In the case of trivial nebentypus (the one we are dealing with) it is easy to see that inner twists must arise from a quadratic Dirichlet character; and Ribet proved that they do not appear when $N$ is square free \cite[3.9]{Ribet1980}.

In particular, the results in \cite{Ribet1985} imply $(P2)$ when $k=2$, $\No$ is a prime, $p\nmid 6(\No-1)$, $\K_f$ is unramified at $p$  and $\calo_f$ is generated over $\ZZ$ by the coefficients $a_n$. 

Some explicit examples with $k=2$ of such modular forms satisfying also $(P0)$ and $(P1)$ are the following:
\begin{enumerate}[-]
\item $N=73$, $p=59$, $\K_f=\QQ(\sqrt{13})$ (form 73.2.a.c in \cite{LMFDB2023}).
\item $N=167$, $p=11$, $\K_f=\QQ(\sqrt{5})$ (form 167.2.a.a in \cite{LMFDB2023}). 
\item $N=383$, $p=13$, $\K_f=\QQ(\sqrt{5})$ (form 383.2.a.a in \cite{LMFDB2023}. 
\item $N=151$, $p=41$, $[\K_f:\QQ]=3$ (form 151.2.a.a in \cite{LMFDB2023}).
\end{enumerate}
Some other examples in weight $k>2$ with $a_p=0$, can be produced when $\KK=\QQ$. In this case in fact conditions (i),(ii), (iii) in \cite[page 192]{Ribet1985} are automatically fulfilled, and if $a_p=0$, then condition (iv) is also satisfied, because the residual representation $\bar\rho_v$ is irreducible, for every $v|p$. There are several  examples of this type with $k=4$, for example
\begin{enumerate}[-]
\item $N=186$, $p=11$, (form 186.4.a.a in \cite{LMFDB2023}).
\item $N=210$, $p=11$ or $p=23$, (form 210.4.a.e in \cite{LMFDB2023}). 
\item $N=1265$, $p=53$,  (form 1265.4.a.c in \cite{LMFDB2023}). 
\end{enumerate}
Some other examples with $k\geq 6$ are:
\begin{enumerate}[-]
\item $N=390$, $p=7$, $k=6$ ((form 390.6.a.c in \cite{LMFDB2023}).
\item $N=66$, $p=5$, $k=8$ (form 66.8.a.a in \cite{LMFDB2023}).
\end{enumerate}
We were not able to find examples with weight $>8$.\par Indeed, the simultaneous occurrence of conditions $(P0)$ and $(P1)$ for a non CM modular form is expected to be  a fairly rare phenomenon. In this direction, Calegari and Sardari \cite{CalegariSardari2021} proved that, fixed a pair of coprime integers $p,N$, there is only a finite number of eigenforms for $\Gamma_1(N)$ having $a_p=0$.\\
%\end{remark}

\noindent{\bf Plan of the paper.}~\par

In Section~\ref{sect:preliminary} we introduce most of the notations we need and we make some preliminary remarks.

 Section~\ref{sect:dio} is devoted to two general statements on lower bounds for the height. The proofs rely on the techniques developed in~\cite{AmorosoDvornicich2000}, \cite{AmorosoZannier2000}, \cite{Habegger2013} and \cite{AmorosoDavidZannier2014}. We fix two (possibly infinite) Galois extensions $\F/\Q$, $\L/\Q$. We fix a rational prime $p$. Usually, lower bounds for the height come from a dicotomy depending on whether $p$ is ramified or not in $\L$. In the second case a Frobenius argument is generally enough to conclude. In the first case we need a descent argument. Here we change the point of view, assuming $\F$ not ramified at $p$, with the residual degrees at $p$ not growing from $\F$ to $\L$. Moreover, we avoid  the equidistribution argument of~\cite{Habegger2013}, by using instead the acceleration lemma of~\cite{AmorosoDavidZannier2014}. This simplifies the proofs, although it leads to worse lower bounds, as remarked in~\cite[Section 4.2.2]{Frey2018}. As a toy example we deduce a generalization (Theorem~\ref{thm:abeliano+})  of the main result of~\cite{AmorosoDvornicich2000}.

One of the crucial steps in applying the height-machinery above is to have a precise as possible description of the sequence of the local higher ramification groups. This is done in~\cite[Lemma 3.3]{Habegger2013} using the fact that the formal group associated to the elliptic curve is a Lubin-Tate module over $\Z_p$.
In Section~\ref{sect:crystalline} we show that, by local Langlands correspondence and the classification of the two dimensional crystalline representations of $G_{\QQ_p}$ given by Breuil, our local representation at $p$ is a twist of a power of a fundamental character, so that Lubin Tate theory is also applicable in  our more general case.

Section~\ref{sect:normal-closure} is devoted to prove Proposition \ref{prop:H.lemma6.2}, which is
  the analogous of Lemma 6.2 of \cite{Habegger2013} on the normal closure of the last ramification group in the field cut out by the reduction modulo $ p^n$ of our modular representation. This was possible thanks to condition $(P2)$ in Assumption \ref{ass:main}. For the proof we needed a group theoretical study of the linear group  of a finite product of local rings. This is done in Appendix~\ref{linear-groups}.
 
The concluding Section~\ref{sect:proofs} contains the proof of the results announced in the introduction.

\section{Notations and preliminary results} 
\label{sect:preliminary}
In the following 
\begin{enumerate}[-]
\item local fields will be denoted by $F,K,E,L,\ldots$
\item global fields will be denoted by $\mathbb{F},\mathbb{K},\mathbb{E},\mathbb{L},\ldots$
\end{enumerate}
(It should be noted in particular that $\FF$ denotes a global field rather than a finite one. On the other hand, in accordance with standard notation, the finite field with $q$ elements will be denoted by $\FF_q$. This should not cause any confusion to the reader.)

For a field $\kappa$, $\overline{\kappa}$ denotes an algebraic closure of $\kappa$, and $G_\kappa=\Gal(\overline{\kappa}/\kappa)$ the absolute Galois group of $\kappa$, endowed with the Krull topology. 
\par 
 Let $\Omega$ be a separated topological group and $\rho:G_\kappa\to \Omega$ be a continuous homomorphism; we shall denote by $\kappa(\rho)=\overline{\kappa}^{\ker \rho}$ the {\sl field cut out} by $\rho$. Then $\kappa(\rho)/\kappa$ is a Galois extension whose Galois group is isomorphic to ${\rm Im}(\rho)$. Moreover, since $\ker{\rho}$ is closed,  $G_{\kappa(\rho)}=\ker\rho $.\par
When $\kappa$ is a local or global field and $v$ is a place of $\kappa$, we shall say that $\rho$ is {\sl unramified} (resp. {\sl totally ramified}) at $v$ if $\kappa(\rho)$ is. Note that $\rho$ induces an isomorphism $\Gal(\kappa(\rho)/\kappa)\rightarrow {\rm Im}(\rho)$ which as usual we denote by the same letter~$\rho$. 

From now on we fix a normalized  eigenform $f=\sum_{n}a_nq^n\in S_k(\Gamma_0(\No))$ for the Hecke algebra of weight $k\geq 2$; then $k$ is an even integer.  We recall the relevant definitions. We denote by $\K$ the number field generated by the $a_n$'s and by $\calo$ its integer ring.   A construction due to Shimura and Deligne  associates to $f$ a family of compatible Galois representations: 
$$
\rho_v:G_\Q \longrightarrow \GL_2(\calo_v),\quad v \hbox{ finite place of } \K.
$$
Each $\rho_v$ is characterized by the following two properties  (where we denote by $p_v$ the rational prime divided by $v$):
\begin{multline}
\label{car1}
\hbox{$\rho_v$ is continuous and unramified at primes $\ell$ with $\ell\nmid\No p_v$;}
\end{multline}%~\vskip -1.4cm
\begin{multline}
\label{car2}
\hbox{for $\ell\nmid Np_v$, the characteristic polynomial of the image by $\rho_v$ of}\\ \quad\quad \quad \hbox{a Frobenius element at $\ell$  is  $X^2-a_\ell X+\ell^{k-1}$.}\hfill
\end{multline}

\begin{remark}
\label{rem:chebotarev}
Let $p$ be a prime and $v$ be any place of $\K$ dividing $p$. Then, by~\eqref{car2} and by Chebotarev lemma, the determinant of $\rho_v$ is $\e_p^{k-1}$, where $\e_p\colon G_\Q\rightarrow\Z_p^\times$ is the $p$-adic cyclotomic character.  (Indeed, for a Frobenius element $\Frob_\ell$ with $\ell\nmid\No p$, we have: $\det\big(\rho_v(\Frob_\ell)\big)=\ell^{k-1}=\e_p^{k-1}(\Frob_\ell)$.)
\end{remark}

We let 
$$
\rho=\prod_{v}\rho_v:G_\Q\longrightarrow \GL_2(\widehat\calo),
$$
where $\widehat\calo$ is the profinite completion of $\calo$. For a prime $p$ and $n>1$, we also denote by $\rho(p^n)$ the reduction $G_\Q\to \GL_2(\widehat\calo/p^n\widehat\calo)=\GL_2(\calo/p^n\calo)$ of $\rho$.

\begin{remark}
\label{rem:chebotarev2}
Let $p$ be a prime which does not divide $\No$ and assume that  $p=\prod\limits_{v\mid p}{\mathfrak P}_v^{e_v}$ in $\calo$. Let $n\geq 1$ be an integer. Then\\%[0.3cm]
\hbox{ }\hskip 2cm
\begin{tikzpicture}[node distance = 1cm, auto]
\node(GQ){$G_\Q$};
\node(ProdGL2) [right of=GQ, node distance = 4cm]{$\prod\limits_v\GL_2(\calo_v)$};
\node(ProdpGL2) [right of=ProdGL2, node distance = 4cm]{$\prod\limits_{v\mid p}\GL_2(\calo_v/{\mathfrak P}_v^{e_vn})$,};
\node(ProdGL2dett) [above of=ProdGL2, node distance = 1.5cm]{$\GL_2(\widehat\calo)$};
\node(ProdpGL2dett) [above of=ProdpGL2, node distance = 1.5cm]{$\GL_2(\calo/p^n\calo)$};
\draw[->] (GQ) to node {$\rho=\prod\limits_v\rho_v$} (ProdGL2);
\draw[->] (ProdGL2) to node {$\pi$} (ProdpGL2);
\draw[double equal sign distance] (ProdGL2) to node  {} (ProdGL2dett);
\draw[double equal sign distance] (ProdpGL2) to node  {} (ProdpGL2dett);
\draw[->, bend left=35] (GQ) to node {$\rho(p^n)$} (ProdpGL2dett);
\end{tikzpicture}~\\[0.3cm]
where $e_v$ is the ramification index at $v$. Thus, by Remark~\ref{rem:chebotarev}, the determinant of $\rho(p^n)$ is $\e_p^{k-1}\!\!\mod p^n$, where $\e_p$ is the $p$-adic cyclotomic character. In particular the determinant takes values in $(\Z/p^n\Z)^\times$.\\%[0.3cm]
\hbox{ }\hskip 2cm
\begin{tikzpicture}[node distance = 1cm, auto]
\node(GQ){$G_\Q$};
\node(Im) [right of=GQ, node distance = 4cm]{${\rm Im}\rho(p^n)$};
\node(sub) [right of=Im, node distance = 1.2cm]{$\subseteq$};
\node(GL2Opn) [right of=Im, node distance = 2.7cm]{$\GL_2(\calo/p^n\calo)$};
\node(Z/p^nZ) [below of=Im, node distance = 2cm]{$(\Z/p^n\Z)^\times$};
\node(sub) [right of=Z/p^nZ, node distance = 1.2cm]{$\subseteq$};
\node(O/p^nO) [below of=GL2Opn, node distance = 2cm]{$(\calo/p^n\calo)^\times.$};
\draw[->] (GQ) to node {$\rho(p^n)$} (Im);
\draw[->, swap] (GQ) to node {$\e_p^{k-1}\!\!\mod p^n$} (Z/p^nZ);
\draw[->] (Im) to node {{\rm det}} (Z/p^nZ);
\draw[->] (GL2Opn) to node {\rm{det}} (O/p^nO);
\end{tikzpicture}
\end{remark}

We let 
$$
\widehat{G}(\calo/p\calo)=\{\alpha\in \GL_2(\calo/p\calo)\ |\ \det(\alpha)\in(\FF_p^\times)^{k-1}\}.
$$ 
Notice that the image of $\rho(p)$ is clearly contained in $\widehat{G}(\calo/p\calo)$ by Remark~\ref{rem:chebotarev2}.\\

For the reader convenience, we also recall a statement from~\cite[Lemma 2.1]{Habegger2013}.
\begin{lemma}
\label{lem:HabLemma 2.1}
Let $F/\Q_p$ be a finite extension. Let $K$, $L$ Galois extensions of $F$ with $K/F$ totally ramified and $L/F$ finite and unramified. Then the restriction map $G_i(KL/L)\rightarrow G_i(K/F)$ is a well defined isomorphism.
%\begin{enumerate}[(i)]
%\item We have $K\cap L = F$ and thus $\Gal(KL/F )$ is canonically isomorphe to $\Gal(K/F)\times\Gal(L/F)$.
%\item The extension KL/K is unramified of degree $[L : F]$, and the extension $KL/L$ is totally ramified of degree $[K : F]$.
%\item $G_i(KL/L)\simeq G_i(K/F)$.
%%(iii) Say i ≥ −1. If σ ∈ Gal(KL/L) ∩ Gi(KL/F ) then σ|K ∈ Gi(K/F ). Moreover, the induced map Gal(KL/L) ∩ Gi(KL/F) → Gi(K/F) is an isomorphism of groups.
%\end{enumerate}
\end{lemma}

\begin{remark}
In ~\cite[Lemma 2.1]{Habegger2013}(iii) $K/F$ is also assumed finite, but this is unnecessary. Moreover in op.cit. it is stated that  $\Gal(KL/L)\cap G_i(KL/F)\simeq G_i(K/F)$, but $\Gal(KL/L)\cap G_i(KL/F)=G_i(KL/L)$ since $L/F$ is unramified.
\end{remark}

\section{Diophantine results.}
\label{sect:dio}
%In this section we propose two general statements on lower bounds for the height. The proofs relies on the techniques developed in~\cite{AmorosoDvornicich2000}, \cite{AmorosoZannier2000}, \cite{AmorosoDavidZannier2014} and \cite{Habegger2013}.
%
%We fix two (possibly infinite) Galois extensions $\F/\Q$, $\L/\Q$. We fix a rational prime $p$. Usually lower bound for the height comes from a dicotomy depending on whether $p$ is ramified or not in $\L$. In the second case a Frobenius argument is generally enough to conclude. In the first case we need a descent argument. Here we change the point of view, assuming $\F$ not ramified at $p$, with the residual degrees at $p$ not growing from $\F$ to $\L$. Moreover, we avoid  the equidistribution argument of~\cite{Habegger2013}, by using instead the acceleration lemma of~\cite{AmorosoDavidZannier2014}. This simplifies the proofs, although it leads to worse lower bounds, as remarked in~\cite[Section 4.2.2]{Frey2018}.\\

We start with a simple result, which is a special case of ~\cite[Theorem 1.5]{AmorosoDavidZannier2014}. Here we allow finite ramification at $p$ in $\F$.
\begin{proposition}
\label{no-ram}
Let $G=\Gal(\F/\Q)$ and let $Z$ be the center.  
%Let $\sigma$ be a lift of a Frobenius over $p$. We assume that the ramification index of $p$ in $\L$ and the order of $\sigma$ in $G/Z$ are both finite. Then $\F$ has property \B.
We assume that the ramification index of $p$ in $\F$ and the order of a lift of a Frobenius $\sigma$ over $p$ in $G/Z$ are both finite. Then $\F$ has property \B.
\end{proposition}
\begin{proof}
Let $e$ and $r$ be respectively the the ramification index and the order of $\sigma$ in $G/Z$. Let $\F_0/\Q$ be a finite Galois extension and $\alpha\in \F_0$. Since $\sigma^r$ lies in the center of $\Gal(\F_0/\Q)$, we have 
$$
\vert\sigma^r(\beta)-\beta^{p^r}\vert_v\leq p^{-1/e}
$$
for any $\beta\in\O_ {\F_0}$ and for any place $v$ dividing $p$. By~\cite[Lemma 2.1]{AmorosoDavidZannier2014}, there exists a positive integer $\lambda$ which is explicitly bounded in terms of $p$ and of $e$ such that 
$$
\vert\sigma^r(\beta^{p^{\lambda}})-\beta^{p^{r+\lambda}}\vert_v\leq p^{-1}
$$
We now apply~\cite[Lemma 2.2]{AmorosoDavidZannier2014}, taking into account that $\sigma^r(\alpha^{p^\lambda})\neq\alpha^{p^{r+\lambda}}$, since otherwise $\alpha$ would be a root of unity.
\end{proof}

%\begin{remark}
%In the situation of~\cite{Habegger2013} we can take $r=2$.
%\end{remark}

Let $\L$ be a (possible infinite) Galois extension. We now make the first assumption on $\L$ which allows to globalise the metric inequality at ramified primes. We follow some ideas from the proofs of 
\cite[Proposition 4.2]{AmorosoDavidZannier2014} and \cite[ Proposition 6.1]{Habegger2013}.
\begin{assumptions}
\label{ass:H1}
% {\viola We have $\mug_{p^\infty}\subseteq \L$.}
% Let $I\subseteq\Gal(\L/\Q)$ be an inertia group over $p$. Then there exist $\tau\in I\cap Z(\Gal(\L/\Q))$ and $g\in\Z$, such that \begin{equation}\label{eq:H1} \tau(\zeta)=\zeta^g \hbox{ for every } \zeta\in \L\cap \mug_{p^\infty}.\end{equation}
There exists $\tau\in G_\QQ\cap I$, where $I$ is an inertia group\footnote{That is, the inverse limit of inertia groups.}  over $p$, such that  \begin{equation}\label{eq:H1}\varepsilon_p(\tau)=g\quad\hbox{with } g\in\ZZ, g>1,\end{equation}
and $\tau|_\L\in  Z(\Gal(\L/\Q))$. 
\end{assumptions}

The second assumption on $\L$ concerns ramification groups. Although of technical nature, it is quite general to be satisfied in several cases, namely  for the extensions cut out by our Galois representations.  Given a finite Galois extension $\E/\Q$ and $i\geq 0$ we recall that the $i$-th ramification group $G_i$ at a place $v$ over $p$ is the subgroup of $\sigma\in\Gal(\E/\Q)$ such that $\vert\sigma(\gamma)-\gamma\vert_v\leq p^{-(i+1)/e}$ for $\gamma\in\O_ \E$. Here $e$ is the ramification index of $\E$ over $p$. The group $G_i$ is the isomorphic image, by restriction, of the corresponding local $i$-th ramification group of the completed extension $\E_v/\QQ_p$. As it is well known, the $i$-th ramification groups at different places of $\E$ over $p$ are conjugate. 

\begin{assumptions}
\label{ass:H2}
There exist $C_1$, $C_2\geq 1$ with the following properties. There exists a sequence of finite Galois extensions $(\L_n/\Q)_{n\geq 0}$ with 
$$
\Q=\L_0\subsetneq \L_1\subsetneq\cdots\subsetneq \L_n\subsetneq\cdots \subsetneq \L,\quad{\rm and}\quad \L=\bigcup_n\L_n.
$$
such that for $n\geq 1$ the extension $\L_n/\Q$ is ramified over $p$. Moreover, let $e_n\geq1$ be the ramification index, $(G^n_i)_i$ the sequence of ramification groups of $\L_n/\Q$ at a place over $p$. We further denote by $H_n=G^n_{i_n}$ the last non trivial ramification group and $\overline{H_n}$ its normal closure in $\Gal(\L_n/\Q)$.
Then 
$$
\frac{e_n}{i_n+1}\leq C_1,\qquad \vert \overline{H_n}\vert\leq C_2
$$
and
$$
\L_n^{\overline{H_n}}=\L_{n-1}.
$$
\end{assumptions}

Even we are not interested in explicit computations, let us remark that the above assumptions  can be weakened in some cases (as for abelian extensions) %, see Remark~\ref{rem:abeliano} below) 
to get sharper lower bounds.

We can now state the main result of this section. We denote by $\mug$ the set of roots of unity and by $\mug_{p^\infty}$ the set of roots of unity of order a power of $p$ in $\Qbar$. 
\begin{proposition}
\label{prop:machine}
Let $\F/\Q$ and $\L/\Q$ two (possibly infinite) Galois extensions. We suppose that:
\begin{enumerate}[-] 
\item $\F$ is unramified at $p$ and it satisfies property \B;
\item $\L$ satisfies Assumptions ~\ref{ass:H1} and~\ref{ass:H2}.
\end{enumerate}
Then $\F\L$ has  property \B{} as well.
\end{proposition}
\begin{proof}
Let $I\subseteq\Gal(\Qbar/\Q)$ be an inertia group over $p$,  $\tau\in I$ and $g\in\Z$ as in Assumption~\ref{ass:H1}. 
%Remark that  $\tau$ is trivial on $\F$, since $\tau$ is in an inertia group and $\F$ is not ramified at $p$.  

By definition of inertia group and by Assumption~\ref{ass:H1}, there exists a prime ideal ${\mathfrak P}$ of $\F$ over $p$ such that $\tau_{\vert \F}\in I({\mathfrak P}\vert p)$. Since $\F$ is not ramified at $p$, this last group is trivial, thus $\tau$ is trivial on $\F$. Again by Assumption~\ref{ass:H1}, $\tau|_\L\in Z(\Gal(\L/\Q))$. Thus 
\begin{equation}
\label{eq:center}
\tau|_{\F\L}\in  Z(\Gal(\F\L/\Q)).
\end{equation}

Since $\F$ has property \B, there exists $c_0>0$ such that
\begin{equation}
\label{eq:B}
\forall\beta\in \F^\star\backslash\mug,\quad h(\beta)\geq c_0.
\end{equation}

Let $\alpha\in (\F\L)^\star\backslash\mug$. We define  
\begin{equation}
\label{eq:beta}
\beta=\frac{\tau(\alpha)}{\alpha^g}.
\end{equation}
We remark that $\beta\not\in \mug$. Otherwise $gh(\alpha)=h(\alpha^g)=h(\tau(\alpha))=h(\alpha)$ which fails since $\alpha$ is not a root of unity and since $g>1$. Moreover, $h(\beta)\leq(g+1)h(\alpha)$. Thus it is enough to show that there exists a $c>0$ depending only on $p$, $c_0$, $C_1$, $C_2$ such that $h(\beta)\geq c$.

We fix a finite Galois extension $\F_0/\Q$ such that $\F_0\subseteq \F$ and $\beta\in \F_0\L$. Let $\{\L_n/\Q\}_{n\geq 0}$ as in Assumption~\ref{ass:H2}. We define $n\geq0$ as the smallest integer such that $\beta\in \F_0\L_n$. If $n=0$ then $\beta\in \F$ and $h(\beta)\geq c_0$ by~\eqref{eq:B}, since $\beta\not\in\mug$. Thus we can assume $n\geq1$. 

We now fix a place $v$ of $\F_0\L_n$ over $p$. Since $\F_0$ is not ramified at $p$, the ramification index $e_n$ of $v_{|\L_n}$ coincide with the ramification index of $v$. Let $G^n_i\subseteq\Gal(\L_n/\Q)$ and $G'^n_i\subseteq\Gal(\F_0\L_n/\Q)$ be respectively the $i$-th ramification subgroups at $v_{|\L_n}$ and at $v$.  
The restriction map
$
\Gal(\F_0\L_n/\F_0)\rightarrow \Gal(\L_n/\QQ)
$
 sends $H'_n:=G'^n_{i_n}$ bijectively in $H_n=G^n_{i_n}$, by Lemma~\ref{lem:HabLemma 2.1} applied to the completions at $v$ of the involved fields. Therefore  the normal closure $\overline{H'_n}$ of $H'_n$ in $\Gal(\F_0\L_n/\Q)$ is sent bijectively in the normal closure $\overline{H_n}$ of $H_n$ in $\Gal(\L_n/\Q)$. 
 By Assumption~\ref{ass:H2} we have
\begin{equation}
\label{eq:card}
\frac{e_n}{i_n+1}\leq C_1,\qquad \vert \overline{H'_n}\vert\leq C_2.
\end{equation}
Since $\L_n^{\overline{H_n}}=\L_{n-1}$ by the last claim of Assumption~\ref{ass:H2}, Galois theory shows that 
\begin{equation}
\label{eq:chiusura-normale}
(\F_0\L_n)^{\overline{H'_n}}=\F_0\L_{n-1}.
\end{equation}

Thus $\beta\not\in (\F_0\L_n)^{\overline{H'_n}}$ by minimality of $n$. Since $\overline{H'_n}$ is generated by the conjugates of $G'^n_{i_n}$, by Galois theory 
$\beta\not\in (\F_0\L_n)^{\phi G'^n_{i_n}\phi^{-1}}$ for some $\phi\in \Gal(\F_0\L_n/\Q)$. 
Thus there exists $\sigma\in G'^n_{i_n}$ such that $\phi\sigma\phi^{-1}(\beta)\neq\beta$. We fix $\phi$, $\sigma$ with the above property and we let  $\beta'=\phi^{-1}(\beta)$. Thus $\sigma(\beta')\neq\beta'$. Moreover, if $\tilde\phi$ is a lifting of $\phi$ to $\F\L$, we have $\beta'=\tfrac{\tau(\alpha')}{\alpha'^g}$ with $\alpha'=\tilde\phi^{-1}(\alpha)$, by definition~\eqref{eq:beta} and since the restriction of $\tau$ is in the center of $\Gal(\F\L/\Q)$ ({\it cf.}~\eqref{eq:center}). Note also that $\beta'\in \F_0\L_n$ and $h(\beta')=h(\beta)$ since $\beta'$ is a conjugate of~$\beta$. Summing up, we have found $\sigma\in G'^n_{i_n}$ and $\alpha'\in \F\L$ such that 
\begin{equation}
\label{eq:beta'}
\beta'=\frac{\tau(\alpha')}{\alpha'^g}\quad{\rm satisfies}\quad \frac{\sigma(\beta')}{\beta'}\neq1
\end{equation}
and we need to prove that there exists a $c>0$ depending only on $p$, $c_0$, $C_1$, $C_2$ such that $h(\beta')\geq c$. 
\begin{fact*}
$\tfrac{\sigma(\beta')}{\beta'}\not\in \mug_{p^\infty}$.
\end{fact*}
\begin{proof}
Let us assume by contradiction that $\tfrac{\sigma(\beta')}{\beta'}\in \mug_{p^\infty}$. Let $\tilde\sigma$ be a lift of $\sigma$ to $\F\L$. By~\eqref{eq:beta'} and since $\tau$ is central,  
$$
\frac{\sigma(\beta')}{\beta'}
=\frac{\tilde\sigma\tau(\alpha')}{\tau(\alpha')}
\left(\frac{\tilde\sigma(\alpha'^g)}{\alpha'^g}\right)^{-1}
=\frac{\tau(\eta)}{\eta^g}.
$$
with $\eta=\tilde\sigma(\alpha')/\alpha'\in \F\L$. Hence 
\begin{equation*}
gh(\eta)=h(\eta^g)=h(\tau(\eta))=h(\eta),
\end{equation*} 
which implies $h(\eta)=0$ since $g>1$. Thus $\eta$ is a root of unity. Write $\eta$ as $\eta=\eta_1\eta_2$ with $\eta_1\in \mug_{p^\infty}$ and with $\eta_2$ of order not divisible by $p$. Hence $\tfrac{\tau(\eta_1)}{\eta_1^g}=1$ by~\eqref{eq:H1}. Thus $\tfrac{\sigma(\beta')}{\beta'}=\tfrac{\tau(\eta_2)}{\eta_2^g}$ has order not divisible by $p$. But $\tfrac{\sigma(\beta')}{\beta'}\in \mug_{p^\infty}$ and $\tfrac{\sigma(\beta')}{\beta'}\neq1$ by~\eqref{eq:beta'}. Contradiction.
% By Bezout's identity, $\eta_1\in \Q(\eta)\subseteq \F\L$. Thus $\eta_1\in \F\L\cap\mug_{p^\infty}=\L\cap\mug_{p^\infty}$, since $\F$ is not ramified at $p$. {\viola [Falso. Per riparare si può supporre nell'ipotesi~\ref{ass:H2} che $\L$ contenga $\mug_{p^\infty}$. Osserviamo per inciso che questo implica per l'argomento di Philippe che $(\F\L_n)_v\cap\mug_{p^\infty}=(\L_n)_v\cap\mug_{p^\infty}$ per ogni $v|p$]} 
\phantom\qedhere\end{proof}

%From now on we let for short $e=e_n$, $k=k_n$ and $H=\overline{H_n}$.
By the fact above,
\begin{equation}
\label{eq:tecnica}
\hbox{For any $p$-power $p^\lambda$ we have $\sigma(\beta')^{p^\lambda}\neq\beta'^{p^\lambda}$.}
\end{equation}
Since $\sigma\in G'^n_{i_n}$, for any $\gamma\in\O_ {\F_0\L_n}$ we have
$$
\vert\sigma(\gamma)-\gamma\vert_v\leq p^{-(i_n+1)/e_n}.
$$
By~\cite[Lemma 2.1]{AmorosoDavidZannier2014}, there exists a positive integer $\lambda$ which is explicitly bounded in terms of $p$, $C_1$ and $C_2$, such that
\begin{equation}
\label{metric}
\vert\sigma(\gamma^{p^\lambda})-\gamma^{p^\lambda}\vert_v\leq p^{-\vert \overline{H'_n}\vert}.
\end{equation}
We consider the action of $\Gal(\F_0\L_n/\Q)$ on itself by conjugation and its stabiliser $C(\sigma)=\{\tau\in\Gal(\F_0\L_n/\Q)\;\vert\; \tau\sigma\tau^{-1}=\sigma\}$. Let $\tau\in C(\sigma)$. We use~\eqref{metric} with $\tau^{-1}\gamma$ instead of $\gamma$. Since $\sigma\tau^{-1}=\tau^{-1}\sigma$ and $\vert\tau^{-1}(\star)\vert_w=\vert\star\vert_{\tau(v)}$, we get 
$$
\vert \sigma(\gamma^{p^\lambda})-\gamma^{p^\lambda}\vert_w\leq p^{-\vert \overline{H'_n}\vert}
$$
with $w=\tau(v)$ and thus for any $w$ in the orbit $S$ of $v$ under the the action of $C(\sigma)$ on the places above $p$. Using~\cite[Lemma 1]{AmorosoDvornicich2000} as in the proof of~\cite[Lemma 2.1]{AmorosoDavidZannier2014} we find
$$
\vert\sigma(\beta'^{p^\lambda})-\beta'^{p^\lambda}\vert_w
\leq c(w)\max(1,\vert\sigma(\beta')\vert_w)^{p^\lambda}\max(1,\vert\beta'\vert_w)^{p^\lambda},\; \forall w\in S
$$
with $c(w)=p^{-\vert \overline{H'_n}\vert}$. This last inequality also holds for an arbitrary place $w$ of $\L$ not dividing $p$, with 
$$
c(w)=
\begin{cases}
1,& \hbox{if } w\nmid\infty;\;\\
2,& \hbox{if } w\mid \infty.
\end{cases}
$$
Let $d_n= [\F_0\L_n:\Q]$ and $d_{n,w}=[(\F_0\L_n)_w:\Q_w]$ with $w$ a place of $\F_0\L_n$. Using the Product Formula to $\sigma(\beta')^{p^\lambda}-\beta'^{p^\lambda}$, which is $\neq0$ by~\eqref{eq:tecnica}, as in the proof of~\cite[Lemma 2.1]{AmorosoDavidZannier2014} we get: 
\begin{align*}
0 & = \sum_w\frac{d_{n,w}}{d_n}\log\vert \sigma(\beta'^{p^\lambda})-\beta'^{p^\lambda}\vert_w\\
 & \leq \sum_w\frac{d_{n,w}}{d_n}\left(\log c(w)
 +p^\lambda\log\max\{1,\vert\sigma(\beta')\vert_w\}+p^\lambda\log\max\{1,\vert\beta'\vert_w\}\right)\\
& = \left(\sum_{w\mid\infty}\frac{d_{n,w}}{d_n}\right)\log 2
-\left(\sum_{w\in S}\frac{d_{n,w}}{d_n}\right)\vert \overline{H'_n}\vert\log p + 2p^\lambda h(\beta').
\end{align*}
Since $\sum_{w\mid\infty}\frac{d_{n,w}}{d_n}=1$ and since $d_{n,w}=d_{n,v}$ for any place $w\mid p$ of $\L_n$, and thus in particular for any $w\in S$, we get
\begin{equation}
\label{first-bound}
2p^\lambda h(\beta')\geq \frac{d_{n,v}}{d_n}\vert S\vert\cdot\vert \overline{H'_n}\vert\log p-\log 2.
\end{equation}
To conclude, we need a lower bound for $S$. Note that $\sigma\in \overline{H'_n}\lhd\Gal(\F_0\L_n/\Q)$, thus the orbit $O$ of $\sigma$ is contained in $\overline{H'_n}$. Thus:
%, and $\vert H\vert\leq p^4$. Since $H$ is normal in $\Gal(\F(N)/\Q)$, the orbit $O$ is contained in $H$. Thus
$$
\vert C(\sigma)\vert =\frac{d_n}{\vert O\vert} \geq \frac{d_n}{\vert \overline{H'_n}\vert}.
$$
The stabiliser of $v$ under the action of $\Gal(\F_0\L_n/\Q)$ is by definition the decomposition group $D(v\vert p)$ of $v$ over $p$ which has (again by definition) cardinality $d_v$. Thus, 
$$
\vert S\vert 
=\frac{\vert C(\sigma)\vert}{\vert D(v\vert p)\cap C(\sigma)\vert}
\geq \frac{d_n}{d_{n,v}}\vert \overline{H'_n}\vert^{-1}.
$$
From~\eqref{first-bound} we then get $h(\beta')\geq c$ with
\begin{align*}
 c=\frac{\log(p/2)}{2p^\lambda}>0. &
\hfill \qed   
\end{align*} \phantom\qedhere
\end{proof}

%Proof in a test-run case
As a toy example we deduce from Proposition~\ref{prop:machine} the following statement which generalizes the main result of~\cite{AmorosoDvornicich2000}.
\begin{theorem}
\label{thm:abeliano+}
Let $\F/\Q$ be a Galois extension unramified at $p$ with property \B. Then $\F(\mug_{p^\infty})$ has property \B{} as well.
\end{theorem}
\begin{proof}
Let $\L=\Q(\mug_{p^\infty})$. Since  $\L/\Q$ is totally ramified at $p$ and abelian, Assumption~\ref{ass:H1} is trivially satisfied. We consider the filtration given by $\L_n=\Q(\zeta_{p^n})$. Then it is easy to see  that $G^n_i=\Gal(\Q(\zeta_{p^n})/\Q(\zeta_{p^j}))$, for $p^{j-1}\leq i\leq p^{j}-1$ (see \cite[Ch. II, \S 10, Ex. 1]{Neukirch1999}). Thus, in the notation of Assumption~\eqref{ass:H2}, $e_n=(p-1)p^{n-1}$,  $i_n=p^{n-1}-1$ and $\overline{H_n}=G_{i_n}^n$. We then have 
$e_n/(i_n+1)\leq p$ and $\vert \overline{H_n}\vert\leq p$. Clearly, $\L_n^{\overline{H_n}}=\L_{n-1}$. The  Assumption~\ref{ass:H2} therefore holds with $C_1=C_2=p$.
\end{proof}

\section{The local representation at $v\mid p$.}
\label{sect:crystalline}
In what follows, we recall some results on crystalline representations. We fix a prime $p$ and an embedding $\Qbar\hookrightarrow \Qbar_p$, which allows to identify $G_{\Q_p}=\Gal(\Qbar_p/\Q_p)$ with a subgroup of $G_\Q$. Note that this is helpful but harmful in our context, as some results may depend on the chosen embedding. To simplify notations and to avoid future conflicts, we simply denote by $\e$ the $p$-adic cyclotomic character $\e_p$. Let $q=p^2$ and denote by $\Q_q$ the unique unramified quadratic extension of~$\Q_p$. Let $v$ be a place of $\K$ over $p$. The two dimensional irreducible crystalline representations of $G_{\QQ_p}$ have been classified in \cite[Proposition 3.1]{Breuil2003}. They are of the form $V_{k',a,\chi}=V_{k',a}\otimes\chi$, where $k'$ is an integer $\geq 2$, $a$ is an element in the maximal ideal of the valuation ring of $\Qbar_p$, and $\chi$ is a crystalline character, that is the product of an unramified character by an integral power of the cyclotomic character. 
The only isomorphisms between them are $V_{k',a,\chi}\simeq V_{k',-a,\chi\mu_{-1}}$, where $\mu_{-1}$ is the unramified quadratic character of $G_{\QQ_p}$. Up to a twist by a power of the cyclotomic character, we can suppose that $\chi$ is unramified. In this case $V_{k',a,\chi}$ has Hodge-Tate weights $(0,k'-1)$, and the $\varphi$-filtered module $D=D_{k',a,\chi}$ associated to $V_{k',a,\chi}$ is $\Qbar_pe_1+\Qbar_p e_2$ with $\varphi$-action given by
\begin{align*} \varphi(e_1) & = p^{k'-1}\chi(p^{-1})^2 e_2\\  
\varphi(e_2) & = -e_1+a \chi(p^{-1})e_2
\end{align*}
and the filtration is given by
e
$$
Fil^i(D)=\begin{cases} D &\hbox{if }  i\leq 0\\
\Qbar_p e_1 &\hbox{if }  1\leq i\leq k'-1\\
0 &\hbox{if } i\geq k'
\end{cases}
$$
In particular, the characteristic polynomial of $\varphi$ on the associated $\varphi$-filtered module is equal to $X^2-aX+\chi(p^{-1})^2p^{k'-1}$.

Let's come back to our representations $\rho_v$, with $v\mid p$. Let as assume $p\nmid \No$ (Assumption~\ref{ass:main} $(P0)$).  
By~\cite[Theorem 1.2.4 (ii)]{Scholl1990}, ${\rho_v}_{|G_{\Q_p}}$ is crystalline and the characteristic polynomial of $\varphi$ on the associated $\varphi$-filtered module is equal to $X^2-a_pX+p^{k-1}$. Moreover, representations associate to modular forms of weight $k$ have Hodge-Tate weight $(0,k-1)$, (see \cite{Faltings1987}). It follows that the filtered module associated to ${\rho_v}_{|G_{\Q_p}}$ is $D_{k',a,\chi}$ with $k'=k$, $a=a_p$, and $\chi(p)=1$ or $-1$. Now assume also $a_p=0$  (Assumption~\ref{ass:main} $(P1)$). Since $V_{k',a,\chi}\simeq V_{k',-a,\chi\mu_{-1}}$, where $\mu_{-1}$ is the unramified quadratic character of $G_{\Q_p}$, in both cases $\rho_v|_{G_{\QQ_p}}$ is isomorphic to $V_{k,0}$. Taking into account the explicit description of $V_{k,0}$ in~\cite[Prop. 3.2]{Breuil2003} we obtain:

\begin{proposition}
\label{prop:crystalline}
Let us assume~\ref{ass:main} $(P0)$ and $(P1)$. Then for every prime $v$ of~$\calo$ above $p$, ${\rho_v}_{|G_{\Q_p}}$ is isomorphic to the crystalline representation
$$
V_{k,0}
\simeq \left (\mathrm{ind}_{G_{\QQ_q}}^{G_{\QQ_p}} \varepsilon_2^{k-1}\right)\otimes \mu_{\sqrt{-1}},
$$
where:\\
- $\varepsilon_2:G_{\QQ_{q}}\to \Z_q^\times$ is one of the two continuous characters characterized by the property that their composition with the injection $\QQ_{q}^\times \hookrightarrow G_{\QQ_{q}}^{ab}$ (given by local class field theory) is trivial on $p$ and is one of the two natural embeddings when restricted to $\ZZ_{q}^\times$;\\
- $\mu_{\sqrt{-1}}$ is one of the two unramified characters of $G_{\QQ_p}$ sending $\Frob_p$ in $\sqrt{-1}$.
\end{proposition}

We also need the following proposition.

\begin{proposition}
\label{prop:zp} 
Let us assume~\ref{ass:main}  $(P0)$-$(P1)$-$(P3)$. 
Let $v$ be a place of $\K$ over~$p$. Then, up to conjugation by an element of $\GL_2(\calo_v)$, the restriction ${\rho_v}_{|G_{\Q_p}}$ takes values in $\GL_2(\Z_p)$.
\end{proposition}
\begin{proof}
Let $\mathfrak{P}$ be the prime ideal corresponding to $v$. Denote by $\rho_v(\mathfrak{P})$ the reduction $G_\Q\to \GL_2(\calo_v/\mathfrak{P}\calo_v)=\GL_2(\calo/\mathfrak{P})$ of $\rho_v$. Since $p+1\nmid k-1$ (which holds {\sl a fortiori} since $\frac{p+1}2\nmid k-1$ by Assumption~\ref{ass:main} $(P3)$), \cite[Lemma 2.1 (1)]{CalegariSardari2021}) ensures that 
%Since $2\leq k\leq p+1$, by an unpublished result by Fontaine (see \cite[Theorem 2.6]{Edixhoven1992}), 
the restriction of $\rho_v(\mathfrak{P})$ to a decomposition group\footnote{By Proposition~\ref{prop:crystalline} this restriction does not depend on the decomposition group but only on $p$.} at $p$ is absolutely irreducible, so that $\rho_v$ {\sl a fortiori} is absolutely irreducible. By Proposition~\ref{prop:crystalline}, the representations ${\rho_v}_{|G_{\Q_p}}\simeq V(k,0)$ take values in $\GL_2(\QQ_p)$, so that the traces lie in $\ZZ_p$.  By \cite[Theorem 1]{Carayol1991}, up to conjugation by an element of $\GL_2(\calo_v)$, the restriction ${\rho_v}_{|G_{\Q_p}}$ takes values in $\GL_2(\ZZ_p)$.
\end{proof}

\begin{remark} From the proof of Proposition~\ref{prop:crystalline} we see that $\ker{\rho_v}_{|G_{\Q_p}}$ does not depend on $v|p$. 
\end{remark}

Let $p$ be a rational prime satisfying Assumptions~\ref{ass:main} $(P0)$, $(P1)$. By Proposition~\ref{prop:crystalline}, for every prime $v$ of $\calo$ above $p$, ${\rho_v}_{|G_{\Q_p}}$ is a twist by an unramified character of the induced representation of a power of a character arising from the Tate module of a Lubin-Tate formal group over $\ZZ_q$ (see \cite[Théorème I.2.1]{Colmez1993}). This is exactly the same situation occuring for the representation of $G_{\Q_p}$ on the $p$-adic Tate module of an elliptic curve $\El$ over $\Q$ having $p$ as a supersingular prime, and this is the key point of the fundamental Galois theoretic analysis of the $p$-power torsion points on $\El$ in~\cite[Lemma 3.3]{Habegger2013}. The proof of this lemma rests on the theory of Lubin-Tate modules and its relation to local class field theory and thus only depends on the local $p$-adic representation. We provide here a reformulation suitable for our situation, without any reference to the theory of elliptic curves.\\

We first recall some simple facts concerning twists of representations. Let $L$ be a local field, $A$ be an abelian group, $\nu,\omega:G_L\to A$ be homomorphisms. Since $\ker\nu\cap\ker \omega =\ker\nu\omega \cap\ker \omega$, we know that 
$$
L\subseteq L(\nu), L(\omega), L(\nu\omega)\subseteq L(\nu)L(\omega)=L(\nu\omega)L(\omega).
$$
\begin{lemma}
\label{lem:characters}
Assume that $\mathrm{im}(\omega)\subseteq \mathrm{im}(\nu)$ and the extensions $L(\nu)/L$, $L(\omega)/L$ are linearly disjoint; then $\mathrm{im}(\nu\omega)=\mathrm{im}(\nu)$ and 
the restriction 
\begin{equation}
\label{eq:nuomega} 
\Gal(L(\nu)L(\omega)/L(\nu\omega))\rightarrow\Gal(L(\omega)/L)
\end{equation}
is an isomorphism.
\end{lemma}
\begin{proof}
Since the two extensions $L(\nu)/L$ and $L(\omega)/L$ are linearly disjoint, the product of restrictions induces an isomorphism
\begin{equation}
\label{eq:nuomega2} 
\Gal(L(\nu)L(\omega)/L)\simeq  \Gal(L(\nu)/L)\times \Gal(L(\omega)/L).
\end{equation}
This formula implies that the image of $\nu$ restricted to $\Gal(\overline{L}/L(\omega))=\ker \omega$ coincides with $\mathrm{im}(\nu)$. Since $\nu=\nu\omega$ over $\ker\omega$, this prove that $\mathrm{im}(\nu\omega)=\mathrm{im}(\nu)$. We now prove the second assertion.
The restriction~\eqref{eq:nuomega} is clearly injective. Let $\tau\in \Gal(L(\omega)/L))$. Since $\mathrm{im}(\omega)\subseteq \mathrm{im}(\nu)$, there exists an element $\sigma\in \Gal(L(\nu)L(\omega)/L)$ such that $\nu(\sigma)=\omega(\tau)^{-1}$ and $\sigma$ restricts to $\tau$ in $\Gal(L(\omega)/L)$. Then $\nu(\sigma)\omega(\sigma)=1$ so that $\sigma\in \Gal(L(\nu)L(\omega)/L(\nu\omega))$. 
\end{proof}

%When $A=\calo^\times$ is the multiplicative group of a $p$-adic ring, we shall also denote by $\nu_n:G_F\to (\calo/p^n\calo)^\times$ the composition of a homomorphism $\nu:G_F\to \calo^\times$ with the reduction modulo~$p^n$.  

\begin{lemma}
\label{lem:twist} Assume that $\mathrm{im}(\omega)\subseteq \mathrm{im}(\nu)$.
If $\nu$ is totally ramified and $\omega$ unramified with finite image, then
\begin{enumerate}[(i)]
\item $\mathrm{im}(\nu\omega)=\mathrm{im}(\nu)$;
\item $\nu\omega$ is totally ramified.
\item For $i\geq 0$ let $G_i$ denote the $i$-th higher ramification group. Then $$G_i(L(\nu\omega)/L)\simeq G_i(L(\nu)/L).$$
In particular
$$
\Gal(L(\nu\omega)/L)\simeq \Gal(L(\nu)/L).
$$
%\item Let $A=\calo^\times$ be the multiplicative group of a $p$-adic ring. Then for $h\leq  n$
%$$
%\Gal(F(\nu_n\omega_n)/F(\nu_{n-h}\omega_{n-h})\simeq \Gal(F(\nu_n)/F(\nu_{n-h})).
%$$
\end{enumerate}
\end{lemma}
\begin{proof}
A diagram may be of help
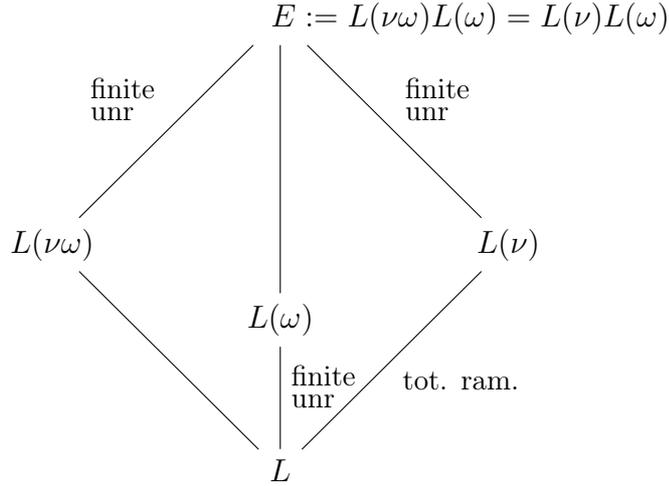
\begin{center}
\[
  \begin{tikzpicture}[node distance = 3cm, auto]
  \node (F) {$L$};
  \node (Fnuomega) [above of=F, left of=F] {$L(\nu\omega)$};
  \node (Fomega) [above of=F, node distance = 2cm] {$L(\omega)$};
  \node (Fnu) [above of=F, right of=F] {$L(\nu)$};
  \node (E) [above of=F, node distance = 6cm] {\hskip 5cm$E:=L(\nu\omega)L(\omega)=L(\nu)L(\omega)$};
  \draw[-] (F) to node {} (Fnuomega);
  \draw[-] (F) to node [swap] {\vbox{\vskip -0.3cm\hbox{\small finite}\vskip -0.2cm\hbox{\small unr}}} (Fomega);
  \draw[-] (F) to node [swap] {\hbox{\small tot. ram.}} (Fnu);
  \draw[-] (Fnuomega) to node{\vbox{\vskip -0.3cm\hbox{\small finite}\vskip -0.2cm\hbox{\small unr}}}  (E);
  \draw[-] (Fomega) to node {} (E);
  \draw[-] (Fnu) to node [swap]{\vbox{\vskip -0.3cm\hbox{\small finite}\vskip -0.2cm\hbox{\small unr}}}  (E);
  \end{tikzpicture}
\]
\end{center}
Here the two extensions on the left top and right top are unramified by \cite[Ch. II, Prop. 7.2]{Neukirch1999}.\\
%Since $\omega$ unramified and has finite order, $F(\omega)/F$ is a cyclic extension.

Let us prove $(i)$ and $(ii)$. Since $\nu$ is totally ramified and $\omega$ is unramified, the two extensions $L(\nu)/L$ and $L(\omega)/L$ are linearly disjoint. Lemma~\ref{lem:characters} then shows that $\im(\nu\omega)=\im(\nu)$ and 
\begin{equation}\label{eq:gradi} 
[E:L(\nu\omega)]=[L(\omega):L].
\end{equation}
Denoting by $f$ the residual degree, we then have:
$$
f(E/L)=f(E/L(\nu))=[E:L(\nu)]\leq [L(\omega):L]
$$
by looking to the right side of the diagram. Moreover, by looking to the left part of the diagram
$$
f(E/L)\geq f(E/L(\nu\omega))=[E:L(\nu\omega)]=[L(\omega):L]
$$
where the last equality comes from~\eqref{eq:gradi}. Thus $f(E/L)=[E:L(\nu\omega)]$ and then 
$L(\nu\omega)/L$ is totally ramified.\\

We now prove $(iii)$. The second assertion follows from the first one for $i=0$. For the first, using (i) and Lemma~\ref{lem:HabLemma 2.1} twice (with $L=L(\omega)$ and with $K=L(\nu)$ the first time, $K=L(\nu\omega)$ the second  time) we get
$$
G_i(L(\nu)/L)=G_i(E/L(\omega))=G_i(L(\nu\omega)/L).
$$
\end{proof}

We are now ready to completely describe the  ramification of the local representation $\rho_p|_{G_{\Q_p}}$, as in \cite[Lemma 3.3]{Habegger2013}.
Recall that $q=p^2$. We put:
$$
%h=k-1,\qquad d=(q-1,k-1),\qquad 
\delta=\frac{q-1}{(q-1,k-1)}. 
$$
Notice that $\delta\not=1$ because $k$ is even.
\begin{lemma}\label{prop:LT} Let $n\geq 1$. We denote by $\varepsilon_{2,n}$ the composition of $\varepsilon_2$ with the reduction modulo~$p^n$. 
\begin{enumerate}[(i)]
\item The extension $\Q_q(\varepsilon_{2,n}^{k-1})/\Q_q$ is totally ramified and abelian of degree $\delta q^{n-1}$, and
$$
\Gal(\Q_q(\varepsilon_{2,n}^{k-1})/\Q_q)\simeq 
\{g^{k-1}\,\vert\, g\in(\Z_q/p^n\Z_q)^\times\}
\simeq \Z/\delta\Z\times \left(\Z/p^{n-1}\Z\right)^2.
$$ 
\item Let $j$ and $i$ be integers with $1\leq j\leq n$ and $q^{j-1}\leq (q-1,k-1)i \leq q^j-1$. The higher
ramification groups are given by
$$G_i(\Q_q(\varepsilon_{2,n}^{k-1})/\Q_q)) = \Gal(\Q_q(\varepsilon_{2,n}^{k-1})/\Q_q(\varepsilon_{2,j}^{k-1})).
$$
\end{enumerate}
\end{lemma}
\begin{proof}
By \cite[Théorème I.2.1]{Colmez1993}, $\varepsilon_2$ is up to isogeny the Tate module of a  Lubin–Tate formal group $\mathcal{F}$ over  $\Z_q$. 
Then $\Q_q(\varepsilon_{2,n})=\Q_q(\mathcal{F}(n))$, where  $\mathcal{F}(n)$ is the group of $p^n$-division points of $\mathcal{F}$. Then in the case $k=2$ (i) and (ii)  follow from \cite[Theorem V.5.4]{Neukirch1999} and \cite[Proposition V.6.1]{Neukirch1999}. 

Assume now $k>2$. The extension $\Q_q(\varepsilon_{2,n}^{k-1})/\Q_q$ is the  subextension of $\Q_q(\varepsilon_{2,n})/\Q_q$ cut out by the $(k-1)$-th power of $\varepsilon_{2,n}$. By assertion (i) in the special case $k=2$, and since $p\nmid k-1$ by Assumption~\ref{ass:main} $(P3)$, we see that $\Gal(\Q_q(\varepsilon_{2,n})/\Q_q(\varepsilon_{2,n}^{k-1}))$ is the subgroup of $\Z/(q-1)\Z$ of index $d:=(q-1,k-1)$. Assertion (i) follows.

We now prove (ii). The group $\Gal(\Q_q(\varepsilon_{2,n})/\Q_q(\varepsilon_{2,n}^{k-1}))$ has order $d$. By Herbrand Theorem \cite[Ch. II, Theorem 10.7]{Neukirch1999}, the ramification group $G_r(\Q_q(\varepsilon_{2,n})/\Q_q)$ is mapped onto $G_{\eta(r)}(\Q_q(\varepsilon_{2,n}^{k-1})/\Q_q)$ by the restriction map 
$$
\mathrm{res}\colon \Gal(\Q_q(\varepsilon_{2,n})/\Q_q) \longrightarrow \Gal(\Q_q(\varepsilon_{2,n}^{k-1})/\Q_q).
$$ 
Here ({\it cf.} \cite[Ch. II, Proposition 10.6]{Neukirch1999})
$$
\eta(r)=\frac 1 {d} \sum_{\sigma\in \Gal(\Q_q(\varepsilon_{2,n})/\Q_q(\varepsilon_{2,n}^{k-1}))}\min\{i(\sigma),r+1\}-1
$$
with $i(\sigma)=v(\sigma(x)-x)$ for a generator $x$ of the ring of integers of $\Q_q(\varepsilon_{2,n})$ over $\Q_q(\varepsilon_{2,n})$.
For $\sigma\in \Gal(\Q_q(\varepsilon_{2,n})/\Q_q(\varepsilon_{2,n}^{k-1}))$ we have 
$$
i(\sigma)=
\begin{cases} 
\infty &\hbox{if } \sigma=id\\
1 &\hbox{ if } \sigma\not=id,
\end{cases}
$$
because $G_1(\Q_q(\varepsilon_{2,n})/\Q_q)$ is a $p$-group and $\Gal(\Q_q(\varepsilon_{2,n})/\Q_q(\varepsilon_{2,n}^{k-1}))$ has order dividing $q-1$  
and thus $G_1(\Q_q(\varepsilon_{2,n})/\Q_q)\cap\Gal(\Q_q(\varepsilon_{2,n})/\Q_q(\varepsilon_{2,n}^{k-1}))$ is trivial. Therefore 
$$
\eta(r)=\frac 1 {d}(r+1+(d-1))-1=\frac r {d}. 
$$
It follows that, for $j$ and $i$  integers such that $1\leq j\leq n$ and $q^{j-1}\leq d\,i \leq q^j-1$, 
\begin{align*} 
G_i(\Q_q(\varepsilon_{2,n}^{k-1})/\Q_q)) 
& = \mathrm{res}\left (G_{d\,i}(\Q_q(\varepsilon_{2,n})/\Q_q)\right )\\ 
&= \mathrm{res}\left (\Q_q(\varepsilon_{2,n})/\Q_q(\varepsilon_{2,j})\right )\quad\hbox{(by the case $k=2$)}\\
& = \Gal(\Q_q(\varepsilon_{2,n}^{k-1})/\Q_q(\varepsilon_{2,j}^{k-1})).
\end{align*}
Now (ii) follows from (i).
\end{proof}

\begin{proposition} Let $n\geq 1$.
\label{lem:3.3}
\begin{enumerate}[(i)]
\item The extension $\Q_q(p^n)/\Q_q$ is totally ramified and abelian of degree $\delta q^{n-1}$. Moreover
$$
\Gal\left(\Q_q(p^n)/\Q_q\right)\simeq \Z/\delta\Z\times \left(\Z/p^{n-1}\Z\right)^2
$$
\item Let $j$ and $i$ be integers with $1\leq j\leq n$ and $q^{j-1}\leq (q-1,k-1)i \leq q^j-1$.The higher
ramification groups are given by
$$
G_i(\Q_q(p^n)/\Q_q) = \Gal(\Q_q(p^n)/\Q_q(p^j)).
$$
In particular, the last non trivial higher ramification group has index 
$$
i_n:=\frac{q^{n-1}-1}{(q-1,k-1)}
$$ 
and 
$$
G_{i_n}=
\begin{cases}
(\Z/p\Z)^2 & \hbox{if } n\geq 2\\ 
\Z/\delta \Z & \hbox{if } n=1.
\end{cases}
$$
\item If $M$ is an integer prime to $p$ and $v\mid p$, then the image of ${\rho_v}_{|G_{\Q_q}}$ contains the scalar matrix $\left(\begin{smallmatrix}M^{k-1} &0\\ 0&M^{k-1} \end{smallmatrix}\right)$.
\end{enumerate}
\end{proposition}
\begin{proof}
By Proposition \ref{prop:crystalline} 
$$
\rho_{v|G_{\Q_q}}\simeq\left (\mathrm{ind}_{G_{\Q_q}}^{G_{\Q_p}} \varepsilon_2^{k-1}\right)\otimes \mu_{\sqrt{-1}}.
$$ 
By Mackey's restriction formula, we see that
\begin{equation}
\label{eq:mackey}
\left (\mathrm{ind}_{G_{\Q_q}}^{G_{\Q_p}} \varepsilon_2^{k-1}\right)_{|G_{\Q_q}}=\varepsilon_2^{k-1} \oplus {\varepsilon'_2}^{k-1},
\end{equation}
where $\varepsilon'_2$ is the character obtained from $\varepsilon_2$ by composing with the algebraic conjugation of $\Q_q/\Q_p$. Let for short $\nu=\varepsilon_{2}^{k-1}$, 
 $\omega= \mu_{\sqrt{-1}|G_{\Q_q}}$ and remark that $\omega$ takes values in $\{\pm1\}$ and thus it is stable by the algebraic conjugation of $\Q_q/\Q_p$. Notice also that $-1\in\im(\nu)=\left(\ZZ_q^\times\right)^{k-1}$, because the weight $k$ is even. Then $\rho_{v|G_{\Q_q}}$ is conjugate to
\begin{equation}
\label{eq:rep}
\begin{aligned}
G_{\Q_{q}}&\longrightarrow \GL_2(\Z_q)\\
\sigma&\to
\left(\begin{smallmatrix}
(\nu\omega)(\sigma)&0\\0&(\nu\omega)'(\sigma)
\end{smallmatrix}\right).
\end{aligned}
\end{equation}
Since the kernel is normal,
\begin{align*}
\ker \rho(p^{n})_{|G_{\Q_{q}}}
&=\{\sigma\in G_{\Q_{q}}\ |\ (\nu\omega)(\sigma)\equiv(\nu\omega)'(\sigma)\equiv 1\pmod{p^n}\}\\
&=\{\sigma\in G_{\Q_{q}}\ |\ (\nu\omega)(\sigma)\equiv 1\pmod{p^n}\}
=\ker(\nu_n\omega_n),
\end{align*}
where we have denoted by $\nu_n,\omega_n$,  the composition of $\nu,\omega$ respectively with the reduction modulo~$p^n$.  
Thus
\begin{equation}
\label{eq:twist}
\Q_q(p^n)=\Q_q(\nu_n\omega_n).
\end{equation}

By its definition, $\nu_n$ is totally ramified; on the other hand $\omega$ is unramified with finite image contained in $\mathrm{im}(\nu)$.
By~\eqref{eq:twist}, by Lemma \ref{lem:twist} (iii) with $L=\Q_q$ and with $\nu$, $\omega$ replaced by $\nu_n$, $\omega_n$ respectively, and by 
Lemma~\ref{prop:LT}
$$
\Gal(\Q_q(p^n)/\Q_q)=\Gal(\Q_q(\nu_n\omega_n)/\Q_q)\simeq\Gal(\Q_q(\nu_n)/\Q_q)\simeq \Z/\delta\Z\times \left(\Z/p^{n-1}\Z\right)^2
$$
which proves (i), and
$$
G_i(\Q_q(p^n)/\Q_q)=G_i(\Q_q(\nu_n\omega_n)/\Q_q)\simeq G_i(\Q_q(\nu_n)/\Q_q)=\Gal(\Q_q(\nu_n)/\Q_q(\nu_j)).
$$
Moreover, again by Lemma \ref{lem:twist} (iii), with $L=\Q_q$ and with $\nu$, $\omega$ replaced by $\nu_i$, $\omega_i$ ($i=j,n$), and by~\eqref{eq:twist}
\begin{align*}
\Gal(\Q_q(\nu_n)/\Q_q(\nu_j))
&=\ker\big(\Gal(\Q_q(\nu_n)/\Q_q)\rightarrow\Gal(\Q_q(\nu_j)/\Q_q)\big)\\
&=\ker\big(\Gal(\Q_q(\nu_n\omega_n)/\Q_q)\rightarrow\Gal(\Q_q(\nu_j\omega_j)/\Q_q)\big)\\
&=\Gal(\Q_q(\nu_n\omega_n)/\Q_q(\nu_j\omega_j))
=\Gal(\Q_q(p^n)/\Q_q(p^j).
\end{align*}
The last two displayed lines prove (ii). 

To prove (iii), we recall that $\nu\omega\colon G_{\Q_{q}}\rightarrow(\Z_q^{\times})^{k-1}$ is surjective by Lemma~\ref{lem:twist} (i). 
Scalar matrices with coefficients in $(\Z_p^\times)^{k-1}$ are central elements, invariant by the algebraic conjugation of $\Q_q/\Q_p$ and thus, by~\eqref{eq:rep}, are in the image of $\rho|_{G_{\QQ_q}}$.
\end{proof}

\section{The normal closure lemma.}
\label{sect:normal-closure}
%Let $p$ be a rational prime and let for short $H_n=\Gal(\Q_q(p^n)/\Q_q(p^{n-1}))$ and $\Omega_n=Gal(\Q(p^n)/\Q)$. Then
\begin{proposition}
\label{prop:H.lemma6.2}  
Let $p$ be a rational prime satisfying Assumptions~\ref{ass:main}. The normal closure of $\Gal(\Q_q(p^n)/\Q_q(p^{n-1}))$ in $\Gal(\Q(p^n)/\Q)$ is $\Gal(\Q(p^n)/\Q(p^{n-1}))$.
\end{proposition}

\begin{proof} 
We closely follow the proof of~\cite[Lemma 6.2]{Habegger2013}, using the result in 
Appendix~\ref{linear-groups} instead of ~\cite[Lemma 6.1]{Habegger2013}.

Since $\Gal(\Q_q(p^n)/\Q_q(p^{n-1}))\subseteq \Gal(\Q(p^n)/\Q(p^{n-1}))$, and since the latter is normal, the normal closure of $\Gal(\Q_q(p^n)/\Q_q(p^{n-1}))$ in $\Gal(\Q(p^n)/\Q)$ is contained in $\Gal(\Q(p^n)/\Q(p^{n-1}))$. We want to prove equality.

Let for short\\[0.3cm]
\begin{tikzpicture}[node distance = 1cm, auto]
\node(Hn){$H_n$};
\node(Hndett)[below of=Hn, node distance = 1cm]{$\Gal(\Q_q(p^n)/\Q_q(p^{n-1}))$};
\node(Omegan) [right of=Hn, node distance = 4cm]{$\Omega_n$};
\node(Omegandett)[below of=Omegan, node distance = 1cm]{$\Gal(\Q(p^n)/\Q(p^{n-1}))$};
\node(Gal) [right of=Omegan, node distance = 3cm]{$\Gal(\Q(p^n)/\Q)$};
\node(GL2) [right of=Gal, node distance = 5cm]{$\GL_2(\calo/p^n\calo)$.};
\draw[right hook->] (Hn) to node {} (Omegan);
\draw[right hook->] (Omegan) to node {} (Gal);
\draw[->] (Gal) to node {$\bar\rho_n=\rho(p^n)$} (GL2);
\draw[double equal sign distance] (Hn) to node  {} (Hndett);
\draw[double equal sign distance] (Omegan) to node  {} (Omegandett);
\end{tikzpicture}~\\[0.3cm]

Since $\calo/p\calo$ is a finite product of local $\F_p$-algebras and $p\geq5$, we can freely use the results of the Appendix~\ref{linear-groups} with $\mA=\calo/p\calo$. We also recall that
$$
\widehat{G}(\calo/p\calo)=\{\alpha\in \GL_2(\calo/p\calo) \ |\ \det(\alpha)\in(\F_p^\times)^{k-1}\}.
$$

Assume first $n=1$. 

By Proposition~\ref{prop:zp}, $\bar\rho_1(H_1)$ is a subgroup of $\GL_2(\F_p)$; by Proposition~\ref{lem:3.3} (ii), its cardinality is
$\delta=\frac{q-1}{(q-1,k-1)}$. Let $s=\frac{p-1}{(p-1,k-1)}$. By Assumption~\ref{ass:main} $(P3)$, $\frac{p+1}2\nmid k-1$. Thus $\frac{p+1}{(p+1,k-1)}>2$ and $\delta>2s$. The determinant maps surjectively $\bar\rho_1(H_1)$ onto $(\F_p^\times)^{k-1}=(\F_p^\times)^{(k-1,p-1)}$, because its image is that  of the $(k-1)$-th power $G_{\QQ_p} \to (\FF_p^\times)^{k-1}$ of the cyclotomic character modulo $p$. Thus $\bar\rho_1(H_1)$ satisfies the hypotheses of Proposition~\ref{prop:norm}, so that the normal closure of $\bar\rho_1(H_1)$ in $\widehat{G}(\calo/p\calo)$ is $\widehat{G}(\calo/p\calo)$ itself. Since by the hypothesis~\ref{ass:main} $(P2)$, $\bar\rho_1$ is an isomorphism between $G_1$ and $\widehat{G}(\calo/p\calo)$, the normal closure of $H_1$ in $G_1$ is $G_1$ itself.\\

Now we turn to the case $n\geq 2$. Let $\sigma\in \Omega_n$. As in the proof of~\cite[Lemma 6.2]{Habegger2013}, $\bar\rho_n(\sigma)$ is the reduction mod $p^n$ of a matrix of the shape $\I+p^{n-1}A_\sigma$, where~$\I$ denotes the $2$-by-$2$ identity matrix and where the \lq\lq log\rq\rq{} matrix $A_\sigma\in M_2(\calo)$ is well-defined mod $p\calo$. Moreover, $\det(\bar\rho_n(\sigma))\in (\Z/p^n\Z)^\times$ by Remark~\ref{rem:chebotarev2}. Since $\det(\I+p^{n-1}A_\sigma)\equiv 1+{\rm trace}(A_\sigma)p^{n-1}\!\!\!\mod p^n$, the matrix $A_\sigma$ has trace in $\F_p$. Let $\widehat M_2(\calo/p\calo)$ be the subgroup of $\M_2(\calo/p\calo)$ consisting of matrices having trace in $\F_p$. We denote by $\mathcal{L}_n$ the map $\Omega_n\rightarrow \widehat M_2(\calo/p\calo)$ which sends $\sigma$ to $A_\sigma$. Thus we have the following commutative diagram~:\\[0.3cm]
%\hbox{ }\hskip 1cm
%\begin{tikzpicture}[node distance = 1cm, auto]
%\node(Omegan){$\Omega_n$};
%\node(Im) [right of=Omegan, node distance = 4cm]{${\rm Im}(\bar\rho_n)$};
%\node(subset) [above of=Im, node distance = 0.5cm]{\rotatebox{90}{$\subseteq$}};
%\node(Mpn) [above of=subset, node distance = 0.5cm]{$\I+p^{n-1}M_2(\calo/p^n\calo)$};
%\node(Mp) [right of=Im, node distance = 4cm]{$\widehat M_2(\calo/p\calo)$};
%\node(Fp) [right of=Mp, node distance = 3cm]{$\F_p$};
%\node(1+Zpn) [below of=Im, node distance = 3cm]{$1+p^{n-1}\Z/p^n\Z$};
%%
%\draw[->] (Omegan) to node {$\bar\rho_n$} (Im);
%\draw[->] (Im) to node {\lq\lq log\rq\rq} (Mp);
%\draw[->] (Mp) to node {trace} (Fp);
%\draw[->, bend left=45] (Omegan) to node {$\mathcal{L}_n$} (Mp);
%\draw[->] (Omegan) to node {\!\!\lq\lq cycl\rq\rq} (1+Zpn);
%\draw[->] (Im) to node {det} (1+Zpn);
%\draw[->] (1+Zpn) to node {\lq\lq log\rq\rq} (Fp);
%\end{tikzpicture}~\\[0.3cm]
%where the two maps \lq\lq log\rq\rq{} send $1+p^{n-1}x \mod p^n$ to $x \mod p$ and where \lq\lq cycl\rq\rq{} is the cyclotomic character. 
\hbox{ }\hskip 2cm
%\begin{equation}
%\label{eq:log}
\begin{tikzpicture}[node distance = 1cm, auto]
\node(Omegan){$\Omega_n$};
\node(Im) [right of=Omegan, node distance = 4cm]{${\rm Im}(\bar\rho_n)$};
\node(subset) [above of=Im, node distance = 0.5cm]{\rotatebox{90}{$\subseteq$}};
\node(Mpn) [above of=subset, node distance = 0.5cm]{$\I+p^{n-1}M_2(\calo/p^n\calo)$};
\node(Mp) [right of=Im, node distance = 4cm]{$\widehat M_2(\calo/p\calo)$};
\draw[->] (Omegan) to node {$\bar\rho_n$} (Im);
\draw[->] (Im) to node {\lq\lq log\rq\rq} (Mp);
\draw[->, bend left=45] (Omegan) to node {$\mathcal{L}_n$} (Mp);
\end{tikzpicture}~\\[0.3cm]
%\end{equation}
Since $\bar\rho_n$ and  \lq\lq log\rq\rq{} are both injective, $\mathcal{L}_n$ is injective as well.

The group $\widehat M_2(\calo/p\calo)$ is a direct sum
$$
\widehat M_2(\calo/p\calo)=\M_2(\calo/p\calo)^0\oplus \F_p\cdot\I
$$
where $\M_2(\calo/p\calo)^0$ is the set of trace zero matrices. We also remark that $\mathcal{L}_n(H_n)\subseteq\M_2(\F_p)$, by Proposition \ref{prop:zp}. As in the proof of~\cite[Lemma 6.2]{Habegger2013}, for $\tau\in \Omega_n$ and $\sigma\in\Gal(\Q(p^{n})/\Q)$ one has 
\begin{equation}
\label{coniugio}
\tilde\rho_n(\sigma)\mathcal{L}_n(\tau)\tilde\rho_n(\sigma)^{-1}=\mathcal{L}_n(\sigma\tau\sigma^{-1}),
\end{equation}
where~\\
\hbox{ }\hskip 0.5cm
\begin{tikzpicture}[node distance = 1cm, auto]
\node(Galpn){$\Gal(\Q(p^n)/\Q)$};
\node(Galp) [right of=Galpn, node distance = 4cm]{$\Gal(\Q(p)/\Q)$};
\node(hG) [right of=Galp, node distance = 4cm]{$\widehat{G}(\calo/p\calo)$};
\node(hGdett) [below of=hG, node distance = 1cm]{$\{\alpha\in \GL_2(\calo/p\calo)\ |\
\det(\alpha)\in(\FF_p^\times)^{k-1}\}$,};
\draw[right hook->] (Galpn) to node {\lq\lq res\rq\rq} (Galp);
\draw[->] (Galp) to node {$\bar\rho_1$} (hG);
\draw[->, bend left=20] (Galpn) to node {$\tilde\rho_n$} (hG);
\draw[double equal sign distance] (hG) to node  {} (hGdett);
\end{tikzpicture}~\\
and \lq\lq res\rq\rq{} denotes the restriction. Let $\overline{H_n}$ be the normal closure of $H_n$ in $\Gal(\Q(p^n)/\Q)$. Recall that $\overline{H_n}\subseteq \Omega_n$. We want to show that $\mathcal{L}_n(\overline{H_n})=\widehat{M_2}(\calo/p\calo)$. This will imply that $\overline{H_n}=\Omega_n$ coincides with $\Omega_n$ as we want (and that $\mathcal{L}_n$ is surjective). Note that, by Assumption~\ref{ass:main} $(P2)$ $\bar\rho_1$ is surjective. Since \lq\lq res\rq\rq{} is obviously surjective, $\tilde\rho_n$ is surjective as well. Thus~\eqref{coniugio} says in particular that $\mathcal{L}_n(\overline{H_n})$ is invariant by the adjoint action\footnote{That is, the action by conjugation.} of $\widehat{G}(\calo/p\calo)$.  Note that $\M_2(\calo/p\calo)^0$ is also invariant by this action. Thus 
$$
V=\mathcal{L}_n(\overline{H_n})\cap \M_2(\calo/p\calo)^0
$$ 
is an $\F_p$-subspace of $\M_2(\calo/p\calo)^0$ invariant by the adjoint action of $\widehat{G}(\calo/p\calo)$, and in particular by that of $\SL_2(\calo/p\calo)\subseteq \widehat{G}(\calo/p\calo)$. Moreover 
$$
V\cap \M_2(\F_p)^0=\mathcal{L}_n(\overline{H_n})\cap\M_2(\F_p)^0\neq\{0\}.
$$ 
Indeed, $\mathcal{L}_n(H_n)$ is contained in $\M_2(\F_p)$ by Proposition \ref{prop:zp}; since $\vert H_n\vert=p^2$ by Proposition~\ref{lem:3.3}$(i)$,  its intersection with $\M_2(\F_p)^0$ is already not trivial.
Then, by Proposition~\ref{prop:irredadj}, $V=\M_2(\calo/p\calo)^0$ and thus $\mathcal{L}_n(\overline{H_n})$ contains $\M_2(\calo/p\calo)^0$. 

The cyclotomic character $H_n\to (1+p^{n-1}\ZZ_p)/(1+p^{n}\ZZ_p)$ is surjective. Indeed by Proposition~\ref{prop:crystalline} and Mackey formula
$$
\rho_v|_{G_{\QQ_q}} \simeq  (\varepsilon_2 \oplus \varepsilon'_2)\otimes \mu_{\sqrt{-1}}|_{G_{\QQ_q}}
$$
(see the proof of Proposition~\ref{lem:3.3}), its image coincides with $N_{\QQ_q/\QQ_p}(1+p^{n-1}\ZZ_q)$ mod $1+p^n\ZZ_p$, and the norm map from $1+p^{n-1}\ZZ_q\rightarrow 1+p^{n-1}\ZZ_p$ is surjective (see for example \cite[Ch. V, \S 2, Prop. 3]{Serre1968}).  By Assumption~\ref{ass:main} $(P3)$
$p\nmid k-1$ and thus the $(k-1)$-th power of the cyclotomic character is surjective. Therefore also  the trace from $\mathcal{L}_n(\overline{H_n})$ to $\F_p$ is surjective.

Thus $\mathcal{L}_n(\overline{H_n})=\widehat{M_2}(\calo/p\calo)$ as claimed.
\end{proof}

\begin{remark}
\label{rem:card}
We have seen in the previous proof that $\overline{H_n}=\Omega_n$ can be identified to a subgroup of $\GL_2(\calo/p\calo)$ if $n=1$ and 
to a subgroup of  $\M_2(\calo/p\calo)$ if $n\geq 2$. Thus $\vert {\overline{H_n}}\vert \leq p^{4[\K:\Q]}$, where we recall that $\KK$ is the number field generated by the coefficients of our modular form.
\end{remark}

\section{Proof of the main results.}
\label{sect:proofs}
Proposition~\ref{no-ram} allows us to shortly settle the unramified case. We define
$$
\rho_p=\prod_{v|p}\rho_v:G_\Q\longrightarrow \GL_2(\widehat\calo)
$$
and 
$$
\rho^{(p)}=\prod_{v\nmid p}\rho_v:G_\Q\longrightarrow\GL_2(\widehat\calo).
$$
\begin{proposition}
\label{teo:unramified}
If Assumption~\ref{ass:main} $(P0)$ and $(P1)$ hold, then $\rho^{(p)}$  has property~\B{}.
\end{proposition}
\begin{proof} Let $\F=\QQ(\rho^{(p)})$. By our assumptions and by~\eqref{car1} and~\eqref{car2}, $\F$ is unramified at $p$ and the characteristic polynomial of $\rho^{(p)}(\Frob_p)$ is $X^2+p^{k-1}$. It follows that $\Frob_p^2$ lies in the center of $\Gal(\F/\Q)$. Then Property (B) for $\F$ is ensured by Proposition~\ref{no-ram}.
\end{proof}

\begin{proof}[Proof of Theorem~\ref{teo:main1}]
We want to apply Proposition~\ref{prop:machine} with $\L=\Q(p^\infty)$. We first check Assumption~\ref{ass:H1}.
By Proposition~\ref{lem:3.3} $(iii)$ (with $M=2$) there exists an element $\tau\in\Gal(\QQ_q(p^\infty)/\QQ_q)$ whose image by $\rho_{v}$ is the scalar matrix $\left(\begin{smallmatrix}2^{k-1} &0\\ 0&2^{k-1} \end{smallmatrix}\right)$. Since $\rho_v|_{G_{\QQ_q}}$ is totally ramified and $\QQ_q/\QQ_p$ is unramified, $\tau$ can be lifted to an element in the inertia subgroup of $G_{\QQ_p}$. As its image is a scalar, $\tau$ lies in the center of $\Gal(\L/\Q)$. Its action on the $p^\infty$-th roots of unity is given by the $p$-adic cyclotomic character $\varepsilon$, and $4^{k-1}=\det(\rho_v(\tau))=\varepsilon(\tau)^{k-1}$. Then ${\epsilon(\tau)}=4\omega$, where $\omega$ is a root of unity in $\ZZ_p^\times$. Up to replace $\tau$ by $\tau^{p-1}$ we obtain  an element that meets the required conditions.

We now check Assumption~\ref{ass:H2}. We consider the filtration given by $\L_n=\Q(p^n)$. Let $n\geq 1$. Fix a place of $\Q(p^\infty)$ over $p$ corresponding to an embedding $\Qbar\hookrightarrow \Qbar_p$. Then, by Proposition~\ref{lem:3.3} $(i)$ and $(ii)$, and in the notation of Assumption~\eqref{ass:H2}, $e_n=\delta q^{n-1}
=\frac{(q-1)q^{n-1}}{(q-1,k-1)}$,  $i_n=\frac{q^{n-1}-1}{(q-1,k-1)}$ and %$G^n_{i_n}
$H_n\simeq\Gal(\QQ_q(p^n)/\QQ_q(p^{n-1}))$. Thus 
$$
\frac{e_n}{i_n+1}=\frac{(q-1)q^{n-1}}{q^{n-1}-1+(q-1,k-1)}\leq q-1.
$$
By Proposition~\ref{prop:H.lemma6.2}, ${\overline{H_n}}=\Gal(\L_n/\L_{n-1})$, so that $\L_n^{\overline{H_n}}=\L_{n-1}$. 
Moreover, by Remark~\ref{rem:card}, $\vert {\overline{H_n}}\vert \leq p^{4[\K:\Q]}$. 

The  Assumption~\ref{ass:H2} therefore holds with $C_1=p^2$ and $C_2=p^{4[\KK:\QQ]}$.
\end{proof}

\begin{proof}[Proof of Theorem~\ref{teo:main}]
Let again $\F=\QQ(\rho^{(p)})$. Then $\F$ is unramified at $p$, and it has Property \B{} by Proposition~\ref{teo:unramified}. By Theorem \ref{teo:main1}, $\F(\rho_p)$ has Property \B. But
\[
\F(\rho_p)=\QQ(\rho^{(p)}\times \rho_p)=\QQ(\rho).
 \qedhere\]
\end{proof}

 For a fixed $n\geq 0$, and a field $\FF\subseteq\Qbar$, consider the condition
\begin{equation} \label{condizione} \hbox{$\FF$ is Galois, unramified at $p$ and $\Frob_p^{n}$ is central in $\Gal(\FF/\Q)$.}
\end{equation}
It is preserved under  composition of fields. Let $\widetilde \FF_{p,n}$ be the maximal extension satisfying \eqref{condizione}. The field $\widetilde \FF_{p,n}$ has \B, by Proposition \ref{no-ram}, and $\widetilde \FF_{p,n}\subseteq \widetilde \FF_{p,m}$ if $n\mid m$. If $f$ fulfils Assumption \ref{ass:main}, then $\QQ(\rho^{(p)})\subseteq \widetilde\FF_{p,2}$, so that, by Proposition \ref{prop:machine},  Theorem \ref{teo:main} has the following
\begin{corollary}\label{cor:genplessis} 
Let $f=\sum_{n\geq 1} a_nq^n \in S_k(\Gamma_0(\No))$ be a normalized eigenform. Assume that there exists a rational prime $p$ satisfying Assumption~\ref{ass:main}. Then for every $n\in\NN$, the field $\widetilde\FF_{p,n}(\rho_f)$ has property \B.
\end{corollary}
Notice that $\widetilde\FF_{p,1}$ contains $\QQ^{\mathrm{tp}}$, the maximal totally $p$-adic extension of $\QQ$. Therefore, when $f$ is a modular form arising from an elliptic curve defined over $\QQ$, our Corollary \ref{cor:genplessis}  intersects to some extent with Theorem 1.6 of \cite{Plessis2024}.

\appendix
\renewcommand{\thetheorem}{\thesubsection.\arabic{theorem}}
\section{On the linear groups of a finite product of local rings}
\label{linear-groups}
Let $\mA=\prod_{i=1}^r \mA_i$ be a finite product of local commutative $\FF_p$-algebras. In this appendix we bring together some  
results on the linear groups overs $\mA$.\\

\label{appendix}
\subsection{A result on the normal closure}
\label{A1}
Let $p$ be a prime number and $k$ be a positive integer. We define
$$
\widehat{G}(\mA)=\{\gamma\in \GL_2(\mA) \ |\ \det(\gamma)\in (\FF_p^\times)^{k-1}\}.
$$
It is the semidirect product $\SL_2(\mA)\rtimes (\FF_p^\times)^{k-1}$, where $x\in(\FF_p^\times)^{k-1}$ acts on $\SL_2(\mA)$ by conjugation by $\left(\begin{smallmatrix}x & 0\\ 0 & 1 \end{smallmatrix}\right)$. The aim of this subsection is to prove the following proposition.
\begin{proposition}
\label{prop:norm} 
Let us assume $p\geq 5$. Let $s=\frac{p-1}{(k-1,p-1)}$ and let $H$ be a subgroup of $\GL_2(\FF_p)$ such that $|H|>2s$ and such that the determinant maps surjectively onto $(\FF_p^\times)^{k-1}$. Then the normal closure of $H$ in $\widehat{G}(\mA)$ is $\widehat{G}(\mA)$ itself.
\end{proposition}
To start with, we remark that since $\mA$ is a semilocal ring, $\SL_2(\mA)$ is generated by elementary matrices, i.e. matrices of the form
\[ S(\alpha)=\begin{pmatrix} 1 & \alpha\\0 &1\end{pmatrix},\quad\quad  T(\alpha)=\begin{pmatrix} 1 & 0\\\alpha &1\end{pmatrix},\quad\quad \alpha\in\mA, \]
(see \cite[Theorem 4.3.9]{HahnOMeara1989}).
\begin{lemma}
\label{lem:quadratigen}
Let us assume $p\geq 5$. Then $\mA$ is generated as a $\FF_p$-vector space by the squares of all its invertible elements.
\end{lemma}
\begin{proof}
Let $\mB$ be the $\FF_p$-vector subspace of $\mA$ generated by the squares of all the invertible elements of $\mA$. Note that $\mB$ is a subring with unity of $\mA$ which contains $\FF_p$. For every $i=1,\ldots ,k$ denote by $e_i\in\mA$ the idempotent projecting to the summand $\mA_i$. Firstly we show that $e_i\in\mB$ for every $i$. Indeed, choose two elements  $a$, $b$ in $\FF_p^\times$ such that $a^2\neq b^2$ (which is possible because $\vert\FF_p^\times\vert>2$) and let $\alpha=a e_i+\sum_{j\not=i}e_j$, $\beta=b e_i+\sum_{j\not=i}e_j$. Then $\alpha^2-\beta^2=(a^2-b^2)e_i$ is a non-zero $\FF_p$-multiple of $e_i$ belonging to $\mB$.

Now we prove that every $\mA_i$  coincides with the subring $\mB_i$ generated over $\FF_p$ by  the squares of the invertible elements. Assume that  
$\alpha\in\mA_i^\times$. Since $p\neq 2$ and since $\mA_i$ is local, one of $\alpha\pm 1$  is in $\mA_i^\times$. Then $(\alpha\pm 1)^2-\alpha^2=\pm 2\alpha+1\in \mB_i$, so that $\alpha\in\mB_i$. It follows $\mA_i^\times\subseteq \mB_i$. Now consider an element $\pi$ in the maximal ideal $\frak{m}_i$ of $\mA_i$. 
then $(1+\pi)^2-1=2\pi+\pi^2=\pi(2+\pi)\in\mB_i$. Since $(2+\pi)^{-1}\in \mA_i^\times\subseteq \mB_i$, we deduce $\pi\in\mB$. Then $\mB_i$ is a subring containing all units of $\mA_i$ and $\mathfrak{m}_i$, so that $\mB_i=\mA_i$.\par
 Finally, assume $\alpha\in\mA$ and $\alpha=\sum_{i=1}^s\alpha_i e_i$ with $\alpha_i\in\mA_i$. For $i=1,\ldots, s$ we can now write $\alpha_i=\sum_ja_{i,j}\beta_{ij}^2$, with $a_{i,j}\in\FF_p$ and $\beta_{ij}\in\mA_i^\times$. Put $\tilde\beta_{ij}=\beta_{ij}e_i+\sum_{r\not=i}e_r$; then $\tilde\beta_{ij}\in\mA^\times$ and
$$
\alpha=\sum_{i,j}a_{ij}\tilde\beta_{ij}^2e_i\in\mB.
$$
\end{proof}
\begin{lemma}
\label{lem:normalclosureinA} 
Let us assume $p\geq 5$. Then the normal closure of $\SL_2(\FF_p)$ in $\SL_2(\mA)$ is $\SL_2(\mA)$ itself.
\end{lemma}
\begin{proof}
We have
\begin{align*}
S(\alpha)S(\beta)=S(\alpha+\beta), \quad S(\alpha)^{-1}=S(-\alpha),\\ 
T(\alpha)T(\beta)=T(\alpha+\beta), \quad T(\alpha)^{-1}=T(-\alpha).
\end{align*}
Let $N$ be the normal closure of $\SL_2(\FF_p)$ in $\SL_2(\mA)$. It suffices to prove that $N$ contains all  matrices $S(\alpha)$ and $T(\alpha)$ with $\alpha$ varying in a set of generators of $\mA$ as a $\FF_p$-vector space.   If $\alpha\in\mA^\times$, then
$$
\begin{pmatrix} \alpha & 0\\ 0 &\alpha^{-1}\end{pmatrix}S(1)\begin{pmatrix} \alpha^{-1} & 0\\ 0 &\alpha\end{pmatrix}=S(\alpha^2),\quad\quad \begin{pmatrix} \alpha & 0\\ 0 &\alpha^{-1}\end{pmatrix}T(1)\begin{pmatrix} \alpha^{-1} & 0\\ 0 &\alpha\end{pmatrix}=T(\alpha^2).
$$
The claim now follows from Lemma \ref{lem:quadratigen}, since $S(1)$, $T(1)$ generate $\SL_2(\FF_p)$.
\end{proof}

\begin{lemma}
\label{lem:normalclosure} 
Let us assume $p\geq 5$. Let $s=\frac{p-1}{(k-1,p-1)}$ and let $H$ be a subgroup of $\widehat G(\FF_p)$ such that $|H|>2s$ and such that the determinant maps surjectively onto $(\FF_p^\times)^{k-1}$. Then the normal closure of $H$ in $\widehat G(\FF_p)$ is $\widehat G(\FF_p)$ itself.
\end{lemma}
\begin{proof} Let $N$ be the normal closure of $H$ in $\widehat G(\FF_p)$. Since $|N|>2s$, the intersection $N\cap \SL_2(\FF_p)$ is a non trivial normal subgroup of $\SL_2(\FF_p)$ of order strictly greater than $2$.  Since $\vert\FF_p\vert\geq4$, the only proper normal subgroup of $SL_2(\FF_p)$ is $\{\pm Id\}$ (see for example \cite[Theorem 8.3 and Lemma 8.6, p.539]{LMFDB2023}). Thus $\SL_2(\FF_p)\subseteq N$. Since the projection over $(\FF_p^\times)^{k-1}$ is surjective, the result follows.
\end{proof}

We can now prove Proposition~\ref{prop:norm}. Let $N$ be the normal closure of $H$ in $\widehat{G}(\mA)$. By Lemma \ref{lem:normalclosure}, $N$ contains $\widehat G(\FF_p)$, so that it contains $\SL_2(\mA)$ by Lemma \ref{lem:normalclosureinA}. The surjectivity of the determinant then shows that $N=\widehat{G}(\mA)$.

\subsection{Irreducibility of the adjoint representation}
\label{A2}
We consider the adjoint representation of $\SL_2(\mA)$, which is realized as its action by conjugation  on the $\FF_p$-vector space $\M_2(\mA)^0$ of $2\times 2$ matrices with coefficients in $\mA$ that have trace zero. We want to prove:

\begin{proposition}
\label{prop:irredadj}
Let us assume $p\geq 3$. Let $V$ be a $\FF_p$-subspace of $\M_2(\mA)^0$ such that $V\cap \M_2(\FF_p)^0\not =\{0\}$ and invariant by the adjoint action of $\SL_2(\mA)$. Then $V=\M_2(\mA)^0$.
\end{proposition}

The following results are well known, we report it for convenience of the reader. Observe that $\M_2(\FF_p)^0$ is generated as an $\FF_p$-vector space by the three matrices
$$
A=\begin{pmatrix} 1 & 0\\ 0 & -1\end{pmatrix},\qquad
B=\begin{pmatrix} 0 & 1\\ 0 & 0\end{pmatrix},\qquad 
C=\begin{pmatrix} 0 & 0\\ 1 & 0\end{pmatrix}.
$$
Moreover if we put
$$
X=\begin{pmatrix} 0 & 1\\ -1 & 0\end{pmatrix},\quad
Y=\begin{pmatrix} 1 & 1\\ 0 & 1\end{pmatrix},\qquad 
Z=\begin{pmatrix} 1 & 0\\ 1 & 1\end{pmatrix},
$$
we see that
\begin{equation}
\label{*}
YAY^{-1}-A=-2B,\quad\quad XBX^{-1}=-C,\quad\quad YCY^{-1}-C+B=A.
\end{equation}

\begin{lemma}
\label{lem:Kirriducibile} 
Let us assume $p\geq 3$. Then the adjoint representation of $\SL_2(\FF_p)$ is irreducible.
\end{lemma}
\begin{proof}
The argument is standard. %See  \url{https://math.stackexchange.com/questions/172240/is-this-adjoint-representation-of-mathfrakgl-2-mathbbf-p-irreducible}]} 
By~\eqref{*}, any invariant subspace containing one of $A,B,C$ will also contain the others. 
Given some arbitrary (non-zero) matrix $M=aA+bB+cC$ of trace 0,  put
$$U=YMY^{-1}-M,\quad\quad V=ZMZ^{-1}-M. $$
Then
$$YUY^{-1}-U=-2cB,\quad\quad ZVZ^{-1}-V=-2bC.$$
It follows that any invariant $\FF_p$-subspace of $   \M_2(\FF_p)^0$ containing $M$ must also contain one of $A,B,C$. The claim follows.
\end{proof}

We can now prove Proposition~\ref{prop:irredadj}. Firstly we notice that $\M_2(\FF_p)^0\subseteq V$ by Lemma \ref{lem:Kirriducibile}. Since $\M_2(\mA)^0=\M_2(\FF_p)^0\otimes_{\FF_p}\mA$, it suffices to show that $V$ is closed under the multiplication by elements in $\mA$.   Let $\alpha$ be an invertible element of $\mA$ and let $B=\begin{pmatrix} 0&1 \\ 0&0\end{pmatrix}$. We have 
$$\begin{pmatrix}\alpha& 0\\ 0 &\alpha^{-1}\end{pmatrix}B\begin{pmatrix}\alpha^{-1}& 0\\ 0 &\alpha\end{pmatrix}=\alpha^{2}B.$$
Therefore $\alpha^2B\in V$, and by Lemma \ref{lem:quadratigen} we get that $\mA B\subseteq V$. Then~\eqref{*} proves that also $\mA A$ and $\mA C$ lie in $V$, so that $\mA V\subseteq V$.

\normalem
\section*{Acknowledgements}
We are grateful to Andrea Conti, Samuel Le Fourn, Davide Lombardo and Arnaud Plessis for helpful discussions.\par
%We thank the referee for a thorough reading of the manuscript and for insightful comments and valuable suggestions, which helped improve the paper.\par
We thank the INdAM GNSAGA for its support.\par 
We acknowledge financial support from the CNRS IEA (International Emerging Action) project PAriAIPP (Problèmes sur l'Arithmétique et l'Algèbre des Petits Points).

\bibliographystyle{abbrv}

\end{document}